\documentclass[12pt]{extarticle}

\usepackage[utf8]{inputenc}
\usepackage{caption}
\usepackage{xfrac} 
\oddsidemargin \evensidemargin
\usepackage{floatrow}
\usepackage{xcolor}
\usepackage{amsmath, amssymb, amsfonts}
\usepackage{mathtools}
\usepackage{enumerate}
\usepackage[all]{xy}
\usepackage[utf8]{inputenc}
\usepackage{bbm}
\usepackage{mathrsfs}
\usepackage{tikz,pgfplots,tikz-3dplot}

\renewenvironment{thebibliography}[1]{
  \begin{oldthebibliography}{#1}
    \setlength{\itemsep}{0.18em}
    \setlength{\parskip}{0em}
}
{
  \end{oldthebibliography}
}

\usepackage[english]{babel}
\usepackage{alphabeta}

\usepackage[colorlinks=true, allcolors=blue,linkcolor=blue, citecolor=blue,backref=page]{hyperref}

\definecolor{mycolor1}{rgb}{0.00000,0.44700,0.74100}
\definecolor{mycolor2}{rgb}{0.8500, 0.3250, 0.0980}
\definecolor{mycolor3}{rgb}{0.9290, 0.6940, 0.1250}
\definecolor{mycolor4}{rgb}{0.4940, 0.1840, 0.5560}
\definecolor{mycolor5}{rgb}{0.4660, 0.6740, 0.1880}

\usepackage{amsthm}
\usepackage{cleveref}

\newtheorem{theorem}{Theorem}[section]

\newtheorem{lemma}[theorem]{Lemma}

\newtheorem{proposition}[theorem]{Proposition}

\theoremstyle{definition}

\newtheorem{definition}[theorem]{Definition}

\newtheorem{examplex}[theorem]{Example}
\newenvironment{example}
{\pushQED{\qed}\examplex}
{\popQED\endexamplex}

\newtheorem{remarkx}[theorem]{Remark}
\newenvironment{remark}
{\pushQED{\qed}\remarkx}
{\popQED\endremarkx}

\numberwithin{equation}{section}

\newtheoremstyle{citing}
{}
{}
{\itshape}
{}
{\bfseries}
{\textbf{.}}
{.5em}
{\thmnote{#3}}
{\theoremstyle{citing}
}

\DeclareMathOperator{\ev}{ev}
\DeclareMathOperator{\ord}{ord}
\DeclareMathOperator{\cH}{\mathcal{H}}
\DeclareMathOperator{\NN}{\mathbb{N}}
\DeclareMathOperator{\CC}{\mathbb{C}}
\DeclareMathOperator{\RR}{\mathbb{R}}
\DeclareMathOperator{\ZZ}{\mathbb{Z}}
\DeclareMathOperator{\QQ}{\mathbb{Q}}

\DeclareMathOperator{\NP}{NP}

\DeclareMathOperator{\Conv}{Conv}
\DeclareMathOperator{\Ann}{Ann}

\DeclareMathOperator{\an}{an}
\DeclareMathOperator{\dR}{dR}
\DeclareMathOperator\Span{Span}
\DeclareMathOperator{\Hom}{Hom}
\DeclareMathOperator{\Res}{Res}

\DeclareMathOperator{\vol}{vol}
\DeclareMathOperator{\GM}{GM}

\DeclareMathOperator{\Sol}{\mathcal{S}ol}

\DeclareMathOperator{\Sing}{Sing}
\DeclareMathOperator{\dlog}{dlog}
\DeclareMathOperator{\im}{im}
\newcommand{\PP}{\mathbb{P}}
\renewcommand{\d}{\mathrm{d}}
\newcommand{\MM}{\mathfrak{M}}
\newcommand{\cM}{\mathcal{M}}
\newcommand{\cN}{\mathcal{N}}
\newcommand{\cL}{\mathcal{L}}

\pgfplotsset{compat=1.18}

\title{Vector Spaces of Generalized Euler Integrals}
\author{Daniele Agostini, Claudia Fevola, Anna-Laura Sattelberger, Simon Telen\\ \bigskip (with an appendix by Saiei-Jaeyeong Matsubara-Heo)}
\date{}

\begin{document}
\maketitle \thispagestyle{empty}

\begin{abstract}
We study vector spaces associated to a family of generalized Euler integrals. Their dimension is given by the Euler characteristic of a very affine variety. Motivated by Feynman integrals from particle physics, this has been investigated using tools from homological algebra and the theory of $D$-modules.
We present an overview and uncover new relations between these approaches. We also provide new algorithmic tools. 
\end{abstract}

\section{Introduction}
In this article, we study vector spaces defined in terms of integrals of the form 
\begin{equation} \label{eq:integrals_intro}
\int_\Gamma  \, f^{s + a} \,x^{\nu + b}\, \frac{\d x}{x} \,\,=\,\, \int_\Gamma \,  \left ( \prod_{j = 1}^\ell f_j^{s_j + a_j} \right) \cdot \left ( \prod_{i = 1}^n x_i^{\nu_i + b_i} \right )\frac{\d x_1}{x_1} \wedge \cdots \wedge \frac{\d x_n}{x_n}.
\end{equation}
Here, $x=(x_1, \ldots, x_n)$ are coordinates on $(\CC^{\ast})^n$ and $f = (f_1, \ldots ,f_\ell)$ denotes a tuple of~$\ell$ Laurent polynomials in~$x.$ We use multi-index notation, i.e., $f^s$ denotes $f_1^{s_1}\cdots f_{\ell}^{s_{{\ell}}},$ and similarly for $x^\nu.$ The integration contour $\Gamma$ is chosen compatibly with the polynomials $f$ so that the integral~\eqref{eq:integrals_intro} converges, which will be made precise in \Cref{sec:deRham}. The exponents~$\nu_i, s_j$ take on complex values, whereas $a_i, b_j \in \ZZ$ are thought of as integer shifts.
Integrals like~\eqref{eq:integrals_intro} were called {\em generalized Euler integrals} by Gelfand, Kapranov, and Zelevinsky in~\cite{GKZ90}.

One motivation for studying generalized Euler integrals comes from particle physics, where they arise as Feynman integrals in the Lee--Pomeransky representation \cite{lee2013critical}. These are evaluated to make predictions for particle scattering experiments.  Varying the integers $a, b$ gives Feynman integrals for different space-time dimensions, which are known to satisfy linear relations. In fact, there exists a finite set  $\{ (a^{(k)},b^{(k)}) \}_{k = 1, \ldots, \chi} \subset \ZZ^\ell \!\times\!\, \ZZ^n$ such that any integral \eqref{eq:integrals_intro} can be written as a linear combination of the corresponding $\chi$ Euler~integrals, see \Cref{thm:intro} for a precise formulation. These basis integrals are called {\em master integrals} in the physics literature \cite{henn2013multiloop}.

This discussion hints at the fact that the integrals \eqref{eq:integrals_intro}, for varying $(a,b) \in \ZZ^\ell \!\times\!\, \ZZ^n,$ generate a finite-dimensional vector space. We present two ways of formalizing this statement. We start with a homological interpretation. Let $s_j \in \CC,$ $\nu_i \in \CC$  be fixed, generic complex numbers and $f\in \CC[x,x^{-1}]^{\ell}$ fixed, nonzero Laurent polynomials such that none of the $f_j$ is a multiplicative unit in $\CC[x,x^{-1}].$ 
Now consider the complement of the vanishing locus of their product $f_1\cdots f_{\ell}$ in the algebraic torus $(\mathbb{C}^*)^n.$ 
This defines the very affine variety 
\begin{equation} \label{eq:X_intro}
X \, \, = \, \, \{\, x \in (\CC^*)^n \,|\, f_1(x) \cdots f_\ell(x) \neq 0 \, \} \, \, = \, \, (\CC^*)^n \setminus V(f_1 \cdots f_\ell) \,\subset\, (\CC^\ast)^n.
\end{equation} 
Regular functions on $X$ are precisely elements of $\CC[x^{\pm 1},f_1^{-1},\ldots,f_\ell^{-1}]. $
We are interested in the twisted homology of $X$ (cf. \cite[Chapter~2]{aomoto2011theory}), where the twist is defined by the logarithmic differential $1$-form 
\begin{align}\label{eq:omega}
\omega \,\coloneqq \, \dlog (f^sx^{\nu }) \, = \, \sum_{j=1}^\ell s_j \cdot \, \dlog \left( f_j \right) \,+\, \sum_{i = 1}^n \nu_i \cdot  \, \dlog \left(x_i\right) \, \in\,  \Omega_X^1(X) .
\end{align}
For each $(a,b),$ the integral \eqref{eq:integrals_intro} defines a linear map on the space of $n$-cycles. We define the $\CC$-vector~space
\begin{equation} \label{eq:Vgammaintro}
V_\Gamma \, \coloneqq  \, \Span_{\CC} \left \{ [\Gamma] \longmapsto \int_\Gamma \,f^{s + a} \,x^{\nu + b} \, \frac{\d x}{x} \right \}_{(a,b) \, \in \, \ZZ^\ell \! \times \!\, \ZZ^n} \, \subseteq \, {\Hom}_{\CC}(\, H_n(X,\omega) \, , \, \CC \, ). 
\end{equation}
Here, $H_n(X,\omega)$ is the $n$-th homology of the twisted chain complex associated to $\omega$. We recall the definition in Section \ref{sec:deRham}. We point out that we use the subscript in the notation $V_{\Gamma}$ to indicate which parameters the integrals are considered to be a function of. In particular, the vector space $V_{\Gamma}$ does not depend on $\Gamma$. Another way to obtain a vector space by varying $(a,b)$ is to view \eqref{eq:integrals_intro} as a function of $s$ and $\nu.$ We fix an integration contour $\Gamma \in H_n(X,\omega)$ and also keep $f \in \CC[x,x^{-1}]^\ell$ fixed. We obtain the $\CC(s, \nu)$-vector space
\begin{align}
V_{s, \nu} \, \coloneqq  \, \Span_{\CC(s, \nu)} \left \{ (s, \nu) \longmapsto \int_{\Gamma} \, f^{s + a} \,x^{\nu + b} \, \frac{\d x}{x} \right \}_{(a,b)\, \in \, \ZZ^\ell \!\times \!\, \ZZ^n}.
\end{align}
The cycle $\Gamma$ depends on $(s,\nu),$ as we will clarify in Section \ref{sec:Mellin}. The vector spaces $V_{\Gamma}$ and~$V_{s, \nu},$ as well as the relations between their generators, are connected in an intriguing way. This is explored in the present article. In the context of Feynman integrals, the vector space $V_\Gamma$ was studied by Mizera and Mastrolia in \cite{mizera2018scattering,mastrolia2019feynman}, and $V_{s, \nu}$ for the case of a single polynomial~$f$ is the central object in the article \cite{FeynmanAnnihilator} by Bitoun, Bogner, Klausen, and Panzer. Variations of~$V_{s, \nu},$ in which the Feynman integral depends on some extra physical parameters, appear~in~\cite{bohm2018complete}.
In~\cite{Lauricella}, the authors investigate hypergeometric integrals, viewed as functions of their parameters, in a motivic context. The connection to Morse theory and intersection theory in \cite{frellesvig2021decomposition,frellesvig2019vector} provides methods to compute linear relations among master integrals.

The motivation in \cite{GKZ90} for studying generalized Euler integrals comes from the theory of GKZ systems. It turns out that \eqref{eq:integrals_intro} gives a natural description of the stalk $V_{c^*}$ of the solution sheaf of a certain system of linear PDEs at the point $c^*.$ We set $a = b = 0$ and fix generic complex values for the parameters $s \in \CC^\ell, \nu \in \CC^n.$ We now think of~\eqref{eq:integrals_intro} as a function of the coefficients of the $f_j$ and generate a $\CC$-vector space by varying the integration~contour~$\Gamma$:
\begin{equation} \label{eq:Vc} 
V_{c^*} \, \coloneqq  \, \Span_{\CC} \left \{ c \, \longmapsto \, 
\int_\Gamma \, f(x;c)^{s} \,x^{\nu}\, \frac{\d x}{x} \right \}_{[\Gamma] \in H_n(X,\omega)},
\end{equation}
where $c$ lies in a small neighborhood of $c^*$ and two functions are identified when they agree on a neighborhood of $c^*.$
The space $X$ depends on $c^*,$ as we explain in detail in Section \ref{sec:GKZ}.

Here, the monomial supports of the Laurent polynomials $f_j$ are fixed, and their coefficients are listed in a vector $c \in \CC^A$ of complex parameters. The notation $\CC^A$ indicates that the entries of $c$ are indexed by a set of exponents $A \subset \ZZ^n.$ The vector space $V_{c^*}$ is a subspace of the hypergeometric functions on $\CC^A.$ It consists of local solutions to a GKZ system, which is a $D_A$-ideal later denoted by $H_A(\kappa)$ with $\kappa =(-\nu,s).$ Here $D_A$ is the Weyl algebra whose variables are indexed by $A.$ We recall definitions and notation in Sections~\ref{sec:Mellin} and \ref{sec:GKZ}. A classical result by Cauchy, Kovalevskaya, and Kashiwara in \mbox{$D$-module} theory relates the dimension of $V_{c^*}$ to that of the $\CC(c)$-vector space $R_A/(R_A \cdot H_A(\kappa)),$ which is the quotient of the rational Weyl algebra $R_A$ by the $R_A$-ideal generated by the GKZ system. The vector space $V_{c^*}$ is a relative version of $V_\Gamma$, cf.~\cite{Del}.
The connection between these two vector spaces is also investigated in the works of Matsubara-Heo \cite{Matsubara2020,MatsubaraHeo2021} in a more general setup. In the recent article \cite{MacaulayFeynman}, the authors present a fast algorithm to compute
Macaulay matrices to efficiently derive Pfaffian systems of GKZ systems.

Fixing generic parameters in each context, all vector spaces seen above share the same dimension. Moreover, this dimension is governed by the topology of $X$ in \eqref{eq:X_intro}.

\begin{theorem}\label{thm:intro}
Let $X \subset (\CC^*)^n$ be the very affine variety \eqref{eq:X_intro}, where $f_j$ are Laurent polynomials with fixed monomial supports and generic coefficients. Let $V_\Gamma, V_{s, \nu}, V_{c^*}, H_A(\kappa)$ be as defined above, with generic choices of parameters each. We have 
\[ \dim_{\CC} \left( V_\Gamma \right) \, = \, \dim_{\CC(s,\nu)} \left(V_{s, \nu}\right) \, = \, \dim_{\CC} \left( V_{c^*} \right) \, = \, \dim_{\CC(c)} \left( R_A / (R_A \cdot H_A(\kappa)) \right)\, = \, (-1)^n \cdot \chi(X),\]
where $\chi(X)$ denotes the topological Euler characteristic of $X.$
\end{theorem} 
\Cref{thm:intro} is a corollary of Theorems \ref{thm:VGamma}, \ref{thm:dimVsnumulti}, and \ref{thm:chiGKZ}, which state the result in each context separately. The requirements on ``genericity'' in each context are made precise in \Cref{thm:chiGKZ} and \Cref{rmk:alpha_generic}.
Although the statement of Theorem \ref{thm:intro} in the case of $V_{s,\nu}$ appears in the literature only for $\ell = 1$, the rest of its content summarizes known results. Other than allowing $\ell>1$ for the vector space $V_{s,\nu}$, we also consider it part of our contribution to bring together scattered results in the literature on this important topic, leading to an accessible, complete proof of \Cref{thm:intro}. We discuss how to algorithmically obtain relations between the generators of $V_{\Gamma}$ and $V_{s, \nu},$ which is relevant in practice. We provide new insights in connections between the relations for these two different vector spaces (e.g., Proposition \ref{prop:3.11}), and develop new numerical methods to compute them in the case of $V_{\Gamma}$ (Section \ref{sec:numerical}). We also highlight what genericity means in each context. We provide examples illustrating the theory and demonstrate how to run computations using different software systems. Our setup applies to the generalized Euler integrals in \eqref{eq:integrals_intro}, and is not restricted to Feynman integrals. However, we discuss this special case in several~remarks.

\bigskip
{\bf Outline.}
Our article is organized as follows. In \Cref{sec:deRham}, we recall twisted de Rham cohomology and homology with coefficients in a local system. We study the vector space $V_\Gamma$ and relations between its generators.
\Cref{sec:Mellin} recalls definitions about algebras of differential and difference operators, and revisits the Mellin transform adapted to our setup. This is used to investigate $V_{s,\nu}.$
\Cref{sec:GKZ} presents the GKZ system leading to $V_{c^*},$ and recalls the background. 
\Cref{sec:numerical} presents numerical methods to compute $\chi(X),$ and to find relations among the integrals we study.
\Cref{appendixA} proves a vanishing result on cohomology groups. In \Cref{app:code}, we provide the code used for computations in \Cref{sec:numerical}. It is also made available, together with our other computational examples, on the {\tt MathRepo}~\cite{fevola2022mathrepo} page \href{https://mathrepo.mis.mpg.de/EulerIntegrals}{https://mathrepo.mis.mpg.de/EulerIntegrals} hosted by MPI MiS.

\section{Twisted de Rham cohomology}\label{sec:deRham}
Throughout this section, $s \in \CC^\ell,$ $\nu \in \CC^{n},$ and $f \in \CC[x, x^{-1}]^{\ell}$ are fixed. None of the $f_j$ are zero or a unit in $\CC[x, x^{-1}].$ We focus on the $\CC$-vector space
\begin{equation} \label{eq:Vgamma} 
V_\Gamma \, = \, \Span_{\CC} \left \{ [\Gamma] \longmapsto \int_\Gamma \, f^{s + a}\,x^{\nu + b} \, \frac{\d x}{x} \right \}_{(a,b) \, \in \, \ZZ^\ell \! \times \!\, \ZZ^n} 
\end{equation}
previously introduced. In particular, we explain that it is isomorphic to a (co-)homology vector space. We denote by $I_{a,b}(\Gamma)$ the integral in \eqref{eq:integrals_intro} to stress the dependence on the integer shifts $(a,b) \in \ZZ^\ell \! \times \!\, \ZZ^n$ and on the integration contour~$\Gamma.$ This section views $I_{a,b}(\Gamma)$ as the pairing between the {\em twisted cycle} $\Gamma$ and the {\em co-cycle} $f^a x^b \frac{\d x}{x}.$ We now introduce the relevant (co-)chain complexes, and refer to \cite{aomoto2011theory} for more details.  

Let $X$ be the very affine variety defined in~\eqref{eq:X_intro}. We start by briefly discussing twisted chains, and later switch to co-chains. Since the parameters $s_j$ and $\nu_i$ are complex numbers, $f^s x^{\nu} = f_1^{s_1} \cdots f_\ell^{s_\ell} x_1^{\nu_1} \cdots x_n^{\nu_n}$ is a multi-valued function on $X.$ To compute our integral~\eqref{eq:integrals_intro}, a branch of this function needs to be specified. A twisted chain $\Gamma$ in \eqref{eq:Vgamma} comes with this information: it belongs to the {\em twisted chain group} $C_n(X,{\cal L}_{\omega}^\vee),$ defined as follows. Following the notation of \cite{aomoto2011theory}, let ${\cal L}_\omega^\vee$ be the line bundle on $X$ whose sections are local solutions $\phi$ to the differential equation $\d \phi - \omega \wedge \phi = 0.$ Here $\omega$ is the differential form in \eqref{eq:omega}. 
One checks that these local sections are $\CC$-linear combinations of branches of $f^s x^{\nu}$ (see \Cref{ex:elliptic0}). For $k = 0, \ldots, 2n,$ we define the {\em $k$-dimensional twisted chain group} $C_k(X,{\cal L}_\omega^\vee)$ as the $\CC$-vector space spanned by elements of the form $\Delta \otimes_{\CC}\phi,$ where $\Delta \in C_k(X)$ is a singular chain of dimension $k$ on~$X,$ and $\phi \in {\cal L}_\omega^\vee(U_\Delta)$ is a local section of ${\cal L}_\omega^\vee$ on an open neighborhood $U_\Delta$ of~$\Delta.$ Two such sections are identified if they coincide on some open neighborhood. The tensor sign indicates bilinearity of $C_k(X) \times {\cal L}_\omega^\vee \rightarrow C_k(X,{\cal L}_\omega^\vee)$ in this construction. The {\em twisted boundary operator} $\partial_\omega : C_k(X,{\cal L}_\omega^\vee) \rightarrow C_{k-1}(X,{\cal L}_\omega^\vee)$ naturally remembers the choice of branch: $\partial_\omega(\Delta \otimes_{\CC}\phi) = \partial \Delta \otimes_{\CC}\phi,$ where $\partial: C_k(X) \rightarrow C_{k-1}(X)$ is the usual boundary operator. An element in $C_k(X,{\cal L}_\omega^\vee) \cap \ker ( \partial_\omega)$ is called a {\em $k$-cycle}. We obtain the following chain complex: 
\[ (C_\bullet(X,{\cal L}_{\omega}^\vee), \partial_\omega) \,\colon \,\,  0 \longrightarrow C_{2n}(X,{\cal L}_\omega^\vee) \overset{\partial_\omega}{\longrightarrow} C_{2n-1}(X,{\cal L}_\omega^\vee) \overset{\partial_\omega}{\longrightarrow} \cdots  \overset{\partial_\omega}{\longrightarrow} C_0(X,{\cal L}_\omega^\vee) \longrightarrow 0.\]
Its homology vector spaces  are 
\[ H_k(X,{\cal L}_\omega^\vee) \, \coloneqq \, \ker( C_k(X,{\cal L}_\omega^\vee) \overset{\partial_\omega}{\longrightarrow}C_{k-1}(X,{\cal L}_\omega^\vee)) \, / \, \im (C_{k+1}(X,{\cal L}_\omega^\vee) \overset{\partial_\omega}{\longrightarrow}C_{k}(X,{\cal L}_\omega^\vee)).\]
To simplify the notation, we will write $H_k(X,\omega) \coloneqq H_k(X,{\cal L}_\omega^\vee)$ in what follows.
As indicated in the introduction, the class $[\Gamma]$ in \eqref{eq:Vgamma} lives in the $n$-th homology vector space $H_n(X,\omega).$ In other words, $\Gamma$ is a linear combination of elements of the form $\Delta \otimes \phi$ considered modulo twisted boundaries. Before proceeding with co-chain complexes, we work out an example.
\begin{example}[$\ell = 2, n = 1$] \label{ex:elliptic0}
Let $f = (x-1, x-2) \in \CC[x,x^{-1}]^2,$ $s = (\sfrac{1}{2}, \sfrac{1}{2}
)$ and $\nu = \sfrac{1}{2}.$ The very affine variety $X$ is $\PP^1 \setminus \{0,1,2,\infty\}.$ Let ${\cal C}$ be the elliptic curve given by $\{ y^2z - x(x-z)(x-2z) = 0 \} \subset \PP^2.$ There is a double covering ${\cal C} \setminus \{ \, 4\text{ points} \, \} \rightarrow X$ whose sheets represent the branches of the multi-valued function $f^s x^{\nu} = \sqrt{x(x-1)(x-2)}.$  These are the solutions to the differential equation $( \d - \omega)\phi = 0,$ where
\begin{equation} \label{eq:omegaelliptic}
\omega \,=\,\left( \frac{1}{2(x-1)} + \frac{1}{2(x-2)} + \frac{1}{2x} \right) \d x.
\end{equation}
It is instructive to think of twisted $1$-cycles on $X$ as projections of singular cycles on \mbox{${\cal C} \setminus \{\,4 \text{ points}\, \}.$} This is illustrated in \Cref{fig:torus}. We view the elliptic curve ${\cal C}$ as two copies of $\PP^1,$ glued along the branch cuts between $0,1$ and $2,\infty,$ drawn in black. The green loops in the left part of the figure lift to the standard basis of the first homology of ${\cal C},$ seen on the right. See for instance \cite[Section 1.3.3, p.~27]{bobenko2011introduction}. The dotted part of the cycle encircling $1,\, 2$ lies on a different branch of the covering. The small red loop, even though it is a singular cycle on $X,$ does not lift to a closed loop on~${\cal C}.$ Hence, to obtain a twisted cycle, this loop needs to be run through twice. The reader who is unfamiliar with twisted cycles might appreciate the discussion in \Cref{sec:numerical}, where we discuss how to numerically evaluate $I_{a,b}(\Gamma)$~when~$n = 1.$ 
\end{example}
\begin{figure}
\centering
\includegraphics[width = 12cm]{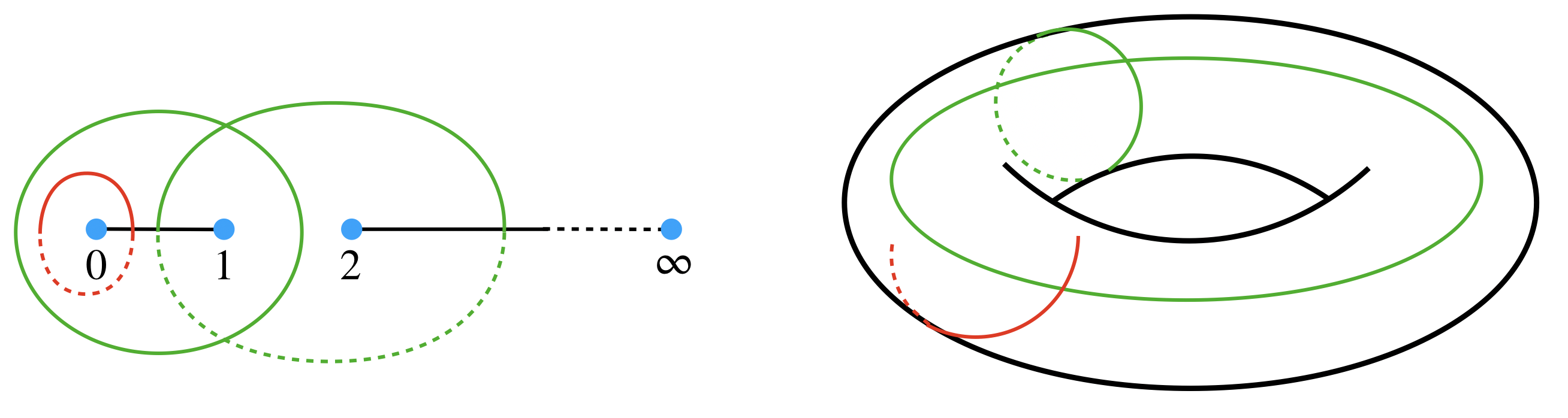}
\caption{Twisted cycles in $\PP^1 \setminus \{0,1,2,\infty \}$ are singular cycles on an elliptic curve.}
\label{fig:torus}
\end{figure}
One then sets up the {\em twisted de Rham complex} dual to the chain complex $C_\bullet(X,{\cal L}_\omega^\vee).$ Since our variety $X$ is affine, it is sufficient to consider the complex of its global sections. This uses the vector spaces ${\cal A}^k(X)$ of smooth $k$-forms on $X.$ The co-chain complex is
\begin{equation} \label{eq:derham} ({\cal A}^\bullet(X), \nabla_\omega)\, \colon \,\, 0 \, \longrightarrow \, {\cal A}^0(X) \, \overset{\nabla_\omega}{\longrightarrow} \, {\cal A}^1(X) \, \overset{\nabla_\omega}{\longrightarrow} \, {\cal A}^2(X ) \, \overset{\nabla_\omega}{\longrightarrow}  \,
\cdots \, \overset{\nabla_\omega}{\longrightarrow} \, {\cal A}^{n}(X) \, \overset{\nabla_\omega}{\longrightarrow} \, 0.
\end{equation}
The {\em twisted differential} is defined as $\nabla_\omega \coloneqq \d + \omega \wedge.$ That is, $\nabla_\omega \colon {\cal A}^k(X) \rightarrow {\cal A}^{k+1}(X)$ is given by $\nabla_\omega(\phi) = \d \phi + \omega \wedge \phi .$ Note that $\omega \in {\cal A}^1(X)$ is indeed a smooth $1$-form on $X.$ For this complex, the cohomology vector spaces are given by
\[ H^k_{\dR}(X,\omega) \, \coloneqq  \, H^k({\cal A}^\bullet(X), \nabla_\omega) \, = \,  \ker \, ({\cal A}^k(X) \overset{\nabla_\omega}{\longrightarrow} {\cal A}^{k+1}(X)) \, / \,  \im \, ({\cal A}^{k-1}(X) \overset{\nabla_\omega}{\longrightarrow} {\cal A}^{k}(X)) . \]
By \cite[Lemma 2.9(1)]{aomoto2011theory}, there is a perfect pairing between twisted homology and cohomology, justifying our claim that the complexes are dual. For $k = n,$ this is given by 
\begin{equation} \label{eq:perfectpairing} \langle \, \cdot \, , \, \cdot  \, \rangle  \colon \, H^n_{\dR}(X,\omega) \times  H_n(X,\omega)  \longrightarrow \CC, \quad  \left( [\phi],[\Gamma] \right) \mapsto \langle [\phi],[\Gamma] \rangle \,\coloneqq \, \int_\Gamma f^s\,x^{\nu} \, \phi.
\end{equation}
This also shows, since $[f^ax^b \frac{\d x}{x}] \in H^n_{\dR}(X,\omega),$ that $I_{a,b}: [\Gamma] \rightarrow I_{a,b}(\Gamma)$ in \eqref{eq:Vgamma} are elements in $\Hom_{\CC}(H_n(X,\omega),\CC).$ This is the inclusion \eqref{eq:Vgammaintro}, which we will show to be an equality.

A drawback of the twisted de Rham complex \eqref{eq:derham} is that its vector spaces ${\cal A}^k(X)$ are not amenable to computations. This is why one uses a twisted {\em algebraic} de Rham complex for~$X,$ whose vector spaces have more explicit descriptions. Moreover, we will see that
this complex has the same cohomology spaces as $({\cal A}^\bullet(X), \nabla_\omega).$ We use the vector spaces $\Omega_X^k(X)$ of {\em regular $k$-forms} on $X.$ An element of $\Omega_X^0(X)$ is a $\CC$-linear combination of $f^{a} x^b,$ where $(a,b)\in \ZZ^{\ell} \times \ZZ^{n}.$ More generally, a regular $k$-form on $X$ is given by $ \sum_{i_1 < \cdots < i_k} h_{i_1,\ldots, i_k} \, \d x_{i_1} \wedge \cdots \wedge \d x_{i_k},$ where $h_{i_1,\ldots, i_k} \in \Omega_X^0(X)$ are global regular functions on~$X.$ This defines the {\em twisted algebraic de Rham complex}
\begin{equation} \label{eq:algderham} (\Omega_X^\bullet(X), \nabla_\omega) \,\colon \,\, 0 \, \longrightarrow \, \Omega_X^0(X) \, \overset{\nabla_\omega}{\longrightarrow} \, \Omega_X^1(X) \, \overset{\nabla_\omega}{\longrightarrow}   \,
\cdots \, \overset{\nabla_{\omega}}{\longrightarrow} \, \Omega_X^{n}(X) \, \overset{\nabla_\omega}{\longrightarrow} \, 0,  
\end{equation}
whose twisted differential is defined as before: $\nabla_\omega = \d + \omega \wedge.$ Here, we view $\omega$ as a regular $1$-form. Note that $\nabla_{\omega}$ defines an integrable connection on $\mathcal{O}_X$ with logarithmic poles on the boundary of a compactification of $X.$
Taking cohomology of~\eqref{eq:algderham} 
gives $\CC$-vector~spaces $$H^k(X,\omega) \, \coloneqq \, H^k(\Omega_X^\bullet(X),\nabla_\omega) \, = \,  \ker \, (\Omega_X^k(X) \overset{\nabla_\omega}{\longrightarrow} \Omega_X^{k+1}(X)) \, / \,  \im \, (\Omega_X^{k-1}(X) \overset{\nabla_\omega}{\longrightarrow} \Omega_X^{k}(X)) .$$ 
\begin{remark}
In passing from the algebraic twisted de Rham complex $\Omega^{\bullet}_X (X)$ to the analytic version ${\cal A}^\bullet(X),$ one actually needs to replace the algebraic variety $X$ by its \emph{analytification}~$X^{\an}.$ This is to emphasize that regular functions should now be thought of in an analytic, rather than algebraic, sense. In this paper, we drop the superscript $(\cdot)^{\an}$ to simplify notation. We write $H_k(X,\omega)$ for the twisted homology of $X^{\an}$ and  distinguish the cohomology of the two cochain complexes by using the subscript $(\cdot)_{\dR}$ to mean the analytic context. 
\end{remark}
By the Deligne--Grothendieck comparison theorem \cite[Corollaire 6.3]{Del}, we have an isomorphism $H^n_{\dR}(X, \omega) \simeq H^n(X, \omega).$
We now relate $H^n(X,\omega)$ to our vector space $V_\Gamma.$
\begin{lemma}\label{lem:isoCvectorspaces}
The $\CC$-vector spaces $V_{\Gamma},$ $ H^n(X,\omega),$ and $H^n_{\dR}(X, \omega)$ are isomorphic.
\end{lemma}
\begin{proof}
Let $H_n(X,\omega)$ be the twisted homology defined above. By the perfect pairing \eqref{eq:perfectpairing} and $H^n_{\dR}(X, \omega) \simeq H^n(X, \omega),$ every $\CC$-linear map $H_n(X,\omega) \rightarrow \CC$ is given by a $\CC$-linear combination of $I_{a,b}$'s. We conclude $V_\Gamma = \Hom_{\CC}(H_n(X,\omega), \CC) \simeq H^n_{\dR}(X, \omega) \simeq H^n(X,\omega).$ 
\end{proof}

By \Cref{lem:isoCvectorspaces}, the dimension $
\dim_{\CC}(V_\Gamma)$ is equal to that of the twisted top cohomology groups.
To relate this to the topology of $X,$ we make use of the following vanishing theorem. 
\begin{theorem} \label{thm:vanishing}
Fix $f \in \CC[x,x^{-1}]^\ell$ and let $X$ be defined as in \eqref{eq:X_intro}. For generic $(s,\nu) \in \CC^{\ell + n}$, we have that $H^\star(X, \omega) = H^\star_{\dR}(X, \omega) = 0$ whenever $\star \neq n.$ 
\end{theorem}
\begin{remark} 
\Cref{thm:vanishing} is proven in work of Douai and Sabbah \cite{DouaiSabbah} for non-degenerate functions with isolated singularities and by Douai~\cite{Douai} for cohomologically tame functions. For lack of a complete reference, a proof for our setup is included in \Cref{appendixA}. The precise meaning of ``generic'' in Theorem \ref{thm:vanishing} follows from that proof, see \Cref{rmk:alpha_generic}.
\end{remark}

Theorem 1.4 in \cite{adolphson1997} is close to \Cref{thm:vanishing}, but makes extra assumptions on the Laurent polynomials $f.$ First, it is required that the Newton polytope $\NP(h) \subset \RR^{n+
\ell}$ of 
\begin{equation} \label{eq:h} h \, \coloneqq \, \sum_{j = 1}^\ell x_{n + j} \, f_j \, \in \, \CC[x_1, \ldots x_{n+\ell}, x_1^{-1}, \ldots, x_{n+\ell}^{-1}] 
\end{equation}
is of maximal dimension $n + \ell -1.$ It is easy to verify that $\NP(h)$ is contained in an affine hyperplane, so its dimension is indeed bounded by $n + \ell -1.$ Second, it is assumed that $h$ is {\em non-degenerate with respect to $\NP(h)$}. This can be phrased as a concrete Zariski open condition on the coefficients of $f_j.$ Namely, it is equivalent to the non-vanishing of the {\em principal $A$-determinant} from \cite[Chapter 10]{gelfand1994discriminants} at $h.$ The exponents $A$ appearing in $h$ will be seen as columns of the matrix $A$ in \Cref{sec:GKZ}. Under these assumptions on $f,$ the integer $\dim_{\CC}(V_\Gamma)$ can be interpreted as the {\em normalized volume} ${\vol}(\NP(h))$ of the Newton polytope~$\NP(h).$ This equals $(n+\ell)!$ times the Euclidean volume of ${\Conv}( \{0\} \cup \NP(h) ).$ 

\begin{theorem} \label{thm:VGamma} Fix $f \in \CC[x,x^{-1}]^\ell$ and let $X$ be defined as in \eqref{eq:X_intro}. For generic $(s,\nu) \in \CC^{\ell + n}$, we have
\begin{equation} \label{eq:dimVGamma}
\dim_{\CC} \left( V_{\Gamma}\right)\,=\, |\chi(X)|.
\end{equation}
Moreover, if $f$ is such that the polynomial $h$ from \eqref{eq:h} is non-degenerate with respect to $\NP(h)$ and $\dim ( \NP(h)) = n + \ell -1,$ then the number \eqref{eq:dimVGamma} equals ${\vol}(\NP(h)).$ 
\end{theorem}

\begin{proof} 
\Cref{lem:isoCvectorspaces} implies $\dim_{\CC}(V_{\Gamma}) = \dim_{\CC} (H^n_{\dR}(X,\omega)) .$ By \Cref{thm:vanishing}, the alternating sum $\sum_{k = 0}^{2n}(-1)^k \dim_{\CC}  (H^k_{\dR}(X,\omega) )$  equals $(-1)^n\dim_{\CC}( H^n_{\dR}(X,\omega)) .$ This coincides with the Euler characteristic $\chi(X)$ by \cite[Theorem 2.2]{aomoto2011theory}. Finally, \Cref{lem:isoCvectorspaces} gives $\dim_{\CC} (H^n(X,\omega) ) = \dim_{\CC}( H^n_{\dR}(X, \omega))$ and the equality $\dim_{\CC} (H^n(X,\omega) ) = {\vol}(\NP(h))$ is \cite[Theorem 1.4]{adolphson1997}.
\end{proof}
We reiterate that the equality \eqref{eq:dimVGamma} holds for arbitrary choices of $f,$ since no extra assumptions are needed in \Cref{thm:vanishing}. The equality $|\chi(X)| = {\vol}(\NP(h))$ needs the assumptions from \cite[Theorem 1.4]{adolphson1997} discussed above. We illustrate this with an example. 
\begin{example}[$n = 2, \ell = 1$] \label{ex:lines}
We consider the very affine surface $X = (\CC^*)^2 \setminus V(f)$ with 
\begin{equation} \label{eq:fcurve}
f \, = \, -xy^2 + 2xy^3 + 3x^2y - x^2y^3 - 2x^3y + 3x^3y^2. 
\end{equation} 
The polynomial $h$ in $n + \ell = 3$ variables is given by $z \cdot f.$ The Newton polytope $\NP(h)$ is the convex hull of the six exponents. It has dimension $n + \ell - 1 = 2,$ and is contained in the plane with third coordinate equal to 1. This is shown in dark blue color in \Cref{fig:polytopes}. Appending the origin results in a hexagonal pyramid with Euclidean volume~$1.$ Multiplying by $(n+\ell)! = 6$ gives ${\vol}(\NP(h)).$ We will compute in \Cref{sec:numerical} that $\chi(X) = 6.$
Now consider
\begin{equation} \label{eq:flines}
f \, = \, -xy^2 + xy^3 + x^2y - x^2y^3 - x^3y + x^3y^2 \, = \, xy(x-1)(y-1)(x-y).
\end{equation} 
The polytope $\NP(h)$ did not change. The very affine variety $X$ is the complement of an arrangement of five lines in $\CC^2,$ shown in \cite[Figure 1]{agostini2021likelihood}. Its Euler characteristic is 2, which is the number of bounded regions in that figure, and the number seen in entry $(k=3,m=4)$ of \cite[Table 1]{agostini2021likelihood}. The discrepancy $2 < 6$ is due to the fact that $h$ is now no longer non-degenerate with respect to $\NP(h).$ The vanishing of the principal ${\cal A}$-determinant at $h$ follows from $h = \frac{\partial h}{\partial x} = \frac{\partial h}{\partial y}
= \frac{\partial h}{\partial z}= 0$ at the point $(x,y,z) = (1,1,1).$
\end{example}

\begin{figure}
    \centering
    \begin{tikzpicture}[baseline=(OO.base),scale=1.6]
    \coordinate (A) at (1,2,1);
    \coordinate (B) at (1,3,1);
    \coordinate (C) at (2,1,1);
    \coordinate (D) at (2,3,1);
    \coordinate (E) at (3,1,1);
    \coordinate (F) at (3,2,1);
    \coordinate (O) at (0,0,0);
    \coordinate (OO) at (0,-0.2,0);

    \draw[fill opacity=0.6,fill = blue,] (A)--(B)--(D)--(F)--(E)--(C)--(A)--cycle;
    \draw[fill opacity=0.2,fill = blue,] (O)--(A)--(C)--(O)--cycle;      
    \draw[fill opacity=0.2,fill = blue,] (O)--(E)--(C)--(O)--cycle;    
    \draw[fill opacity=0.2,fill = blue,] (O)--(E)--(F)--(O)--cycle;     
    \draw[fill opacity=0.2,fill = blue,] (O)--(D)--(F)--(O)--cycle;      
    \draw[fill opacity=0.2,fill = blue,] (O)--(B)--(D)--(O)--cycle;    
    \draw[fill opacity=0.2,fill = blue,] (O)--(B)--(A)--(O)--cycle; 
 
    \end{tikzpicture}
    \qquad \qquad 
    \tdplotsetmaincoords{80}{50}
    \begin{tikzpicture}[baseline=(OO.base),scale=3,tdplot_main_coords]
    \coordinate (A) at (1,1,0);
    \coordinate (B) at (0,1,0);
    \coordinate (C) at (1,0,1);
    \coordinate (D) at (0,0,1);
    \coordinate (O) at (0,0,0);
    \coordinate (OO) at (0,-1.1,0);

    \draw[fill opacity=0.6,fill = blue,] (A)--(B)--(D)--(C)--(A)--cycle;
    \draw[fill opacity=0.2,fill = blue,] (O)--(A)--(C)--(O)--cycle;      
    \draw[fill opacity=0.2,fill = blue,] (O)--(D)--(C)--(O)--cycle;    
    \draw[fill opacity=0.2,fill = blue,] (O)--(D)--(B)--(O)--cycle;     
    \draw[fill opacity=0.2,fill = blue,] (O)--(A)--(B)--(O)--cycle;     
    \end{tikzpicture}

    \caption{The polytopes $\NP(h)$ and ${\Conv}(\{0\} \cup \NP(h))$ from \Cref{ex:lines} (left) and \Cref{ex:elliptic}~(right). }
    \label{fig:polytopes}
\end{figure}
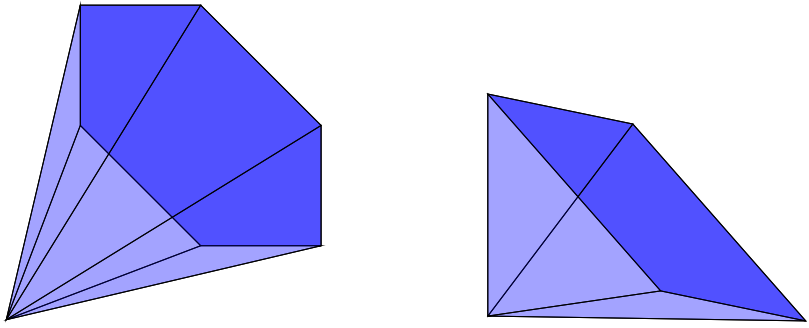

In what follows, we show how the twisted de Rham complex $(\Omega^\bullet_X(X), \nabla_\omega)$ leads to \mbox{$\CC$-linear} relations among the functions $I_{a,b}: [\Gamma] \mapsto I_{a,b}(\Gamma)$ spanning the vector space $V_{\Gamma}.$ In the proof of \Cref{lem:isoCvectorspaces}, we have used the perfect pairing \eqref{eq:perfectpairing} between homology and cohomology. The isomorphism $H^n(X, \omega) \simeq \Hom_{\CC}( \, H_n(X, \omega), {\CC}\,)$
is explicitly given by $[f^a x^b \frac{\d x}{x}] \mapsto I_{a,b}.$ From this we conclude that for all complex constants $C_{a,b}$
\begin{align}\label{eq:relInCohomology}
\begin{split}
\sum_{a,b} C_{a,b} \cdot I_{a,b} \, = \,  0 \text{ in } V_\Gamma & \, \Longleftrightarrow \, \left [ \sum_{a,b} C_{a,b} \cdot  f^a x^b \, \frac{\d x}{x} \right ] \, = \,  0 \text{ in }   H^n(X, \omega)  \\
& \, \Longleftrightarrow \, \, \,  \sum_{a,b} C_{a,b} \cdot f^a x^b \, \frac{\d x}{x} \, \in \, \im \left( \nabla_\omega\right). \end{split}
\end{align}
Hence for any $\phi \in \Omega^{n-1}_X(X),$ we find linear relations between the generators $I_{a,b}$ by expanding 
\begin{align}\label{eq:nablaPhi}
\nabla_{\omega}(\phi) \,=\, \sum_{a,b} \,C_{a,b}(\phi) \cdot  f^a\,x^b \, \frac{\d x}{x}.
\end{align}

\begin{remark}
In the physics literature, relations obtained in this way are commonly referred to as {\em IBP relations}. They are typically derived using Stokes' theorem in the setup of dimensional regularization, see for instance \cite[Proposition 12]{dimreg}. We stress that we deduce our relations exploiting purely cohomological arguments.
\end{remark}

We clarify this procedure with two examples.
\begin{example}[\Cref{ex:elliptic0} continued] \label{ex:elliptic}
Let $f = (x-1, x-2) \in \CC[x,x^{-1}]^2,$ so that $X = \PP^1 \setminus \{0,1,2,\infty \}$ is a punctured Riemann sphere with Euler characteristic $\chi(\PP^1)-4=-2.$ We can obtain this also via a volume computation by setting $h = xy - y + xz - 2z.$ Appending $0$ to $\NP(h)$ gives the pyramid shown in the right part of \Cref{fig:polytopes}. We compute ${\vol}(\NP(h)) = 2 = - \chi(X).$ 
For a concrete example of \eqref{eq:nablaPhi}, we take $s, \nu$ as in \Cref{ex:elliptic0} and set $\phi = 1.$ 
Applying $\nabla_\omega = \d + \omega \wedge : \Omega^0_X(X) \rightarrow \Omega_X^1(X)$ with $\omega$ as in \eqref{eq:omegaelliptic} to $1 \in \Omega^0_X(X),$ we~get
\[ \nabla_\omega(1) \,\,= \,\,  \frac{1}{2} f_1^{-1} f_2^0 x^1 \, \frac{\d x}{x} \,+\, \frac{1}{2} f_1^{0}f_2^{-1}x^1 \, \frac{\d x}{x} \,+\,\frac{1}{2} f_1^0f_2^0x^{0} \, \frac{\d x}{x}.\]
In the notation \eqref{eq:nablaPhi}, this means $C_{(-1,0),1}(1) = C_{(0,-1),1}(1) = C_{(0,0),0}(1) = \sfrac{1}{2}.$ We remind the reader that this entails that, for every choice of the twisted cycle $\Gamma,$ we have 
\[ \int_\Gamma \frac{\sqrt{(x-1)(x-2)x} }{x-1} \, \d x \,+\, \int_\Gamma \frac{\sqrt{(x-1)(x-2)x} }{x-2} \, \d x \,+ \, \int_\Gamma \frac{\sqrt{(x-1)(x-2)x} }{x} \, \d x \,\,=\,\,0. \qedhere \]
\end{example}
In what follows, we denote by $x_{\widehat{k}}\coloneqq x_1\cdots x_{k-1}\cdot x_{k+1}\cdots x_n$ and by $\d x_{\widehat{k}}$ the $(n-1)$-form $\d x_{\widehat{k}}\coloneqq \d x_1 \wedge \cdots \wedge \d x_{k-1}\wedge \d x_{k+1} \wedge \cdots \wedge \d x_n,$ i.e., the (differential with respect to the) $k$-th coordinate is omitted. In this notation,
\begin{align}\label{eq:dxk}
\frac{\d x_{\widehat{k}}}{x_{\widehat{k}}} \, = \, \frac{\d x_1}{x_1} \wedge \cdots \wedge  \frac{\d x_{k-1}}{x_{k-1}} \wedge \frac{\d x_{k+1}}{x_{k+1}}\wedge \cdots \wedge \frac{\d x_n}{x_n} \quad \text{and} \quad \frac{\d x_{\widehat{k}}}{x} \, = \, \frac{\d x_{\widehat{k}}}{x_1\cdots x_n} .
\end{align}

\begin{example}[$\ell = 1$]
Let $f \in \CC[x,x^{-1}]$ be a Laurent polynomial in $n$ variables, and let $X = (\CC^*)^n \setminus V(f)$ be the associated very affine variety. For $1 \leq k \leq n,$ consider the $(n-1)$-form $\phi = p \, f^a x^b \, \frac{\d x_{\widehat{k}}}{x} \in\Omega^{n-1}_X(X),$ with $p \in \CC[x].$ Applying $\nabla_\omega = \d + \omega \wedge$ with $\omega$ as in~\eqref{eq:omega} to $\phi,$ we obtain the following relation in $H^n(X,\omega)$:
\begin{equation}\label{eq:imageNabla}
\left(\frac{\partial p}{\partial x_k}f^a x^b \,+\,(a+s) p \frac{\partial f}{\partial x_k}f^{a-1}x^b \,+\, (b_k+\nu_k-1) p f^a x^{b-e_k}\right) \frac{\d x}{x} \,=\, 0, 
\end{equation}
where $e_k\in \NN^n$ denotes the $k$-th standard unit vector. By expanding $p$ as a sum of monomials, the equivalence in \eqref{eq:relInCohomology} gives a linear relation among the integrals $I_{a,b}.$
\end{example}

For any regular $(n-1)$-form $\phi,$ the method illustrated above provides a relation among some of the $I_{a,b}.$ However, this does not (directly) allow to compute a relation among a set of given generators $\{I_{a,b} \}_{(a,b) \in S},$ where $|S| > \chi(X).$ Doing so---using tools from numerical nonlinear algebra---is one of our topics in \Cref{sec:numerical}.

\section{Mellin transform}\label{sec:Mellin}
Let $f=(f_1,\ldots,f_{\ell})\in \CC[x_1,\ldots,x_n]^{\ell}$ be a tuple of polynomials. This section studies the $\CC(s, \nu)$-vector space
\begin{equation}\label{eq:Vs,nu}
\ V_{s, \nu} \, \coloneqq  \, \Span_{\CC(s, \nu)} \left \{ (s, \nu) \longmapsto \int_{\Gamma} \, f^{s + a} \,x^{\nu + b} \, \frac{\d x}{x} \right \}_{(a,b)\, \in \, \ZZ^\ell \!\times \!\, \ZZ^n},
\end{equation}
defined in the introduction. For the sake of simplicity, throughout this section, we denote the integral in \eqref{eq:integrals_intro} by $I_{a,b}.$ It is to be read as a function of the variables $s\in\CC^\ell$ and $\nu\in \CC^n.$
Hence, the vector space \eqref{eq:Vs,nu} can be equivalently expressed as
\begin{align}
V_{s,\nu} \,=\, \sum_{(a,b)\,\in\, \ZZ^{\ell}\times \ZZ^n} \CC(s,\nu)\cdot I_{a,b}.
\end{align}
We comment on how to interpret the functions $I_{a,b}.$ We here write $\omega(s^\ast,\nu^\ast)$ to emphasize the dependence of $\omega$ from \eqref{eq:omega} on $s^\ast, \nu^\ast.$ For fixed~$(s^*,\nu^*),$ let $\Gamma(s^*,\nu^*)$ be a twisted cycle in $H_n(X,\omega(s^*,\nu^*))$. 
To compute $I_{a,b}(s,\nu)$ for $(s,\nu)$ in a small open neighborhood of $(s^*,\nu^*),$ one integrates over the unique cycle $\Gamma(s,\nu)$ obtained from $\Gamma(s^*,\nu^*)$ by analytic continuation. That is, $I_{a,b}(s,\nu) = \langle [f^ax^b \frac{\d x}{x}], [\Gamma(s,\nu)] \rangle,$ with $\langle \cdot, \cdot \rangle$ as in \eqref{eq:perfectpairing}. This defines $I_{a,b}$ on an open neighborhood of some fixed $(s^*,\nu^*).$ The results in this section do not depend on the choice of $(s^*,\nu^*)$ or the choice of cycle $\Gamma(s^*,\nu^*).$ In what follows, we tacitly assume that these are~fixed. 

Throughout this section, we assume that $f = (f_1, 
\ldots, f_\ell)$ is a tuple of polynomials instead of Laurent polynomials. There is no loss of generality: for some $m\in \NN^n,$ $\tilde{f}\coloneqq x^mf\in \CC[x_1,\ldots,x_n]^\ell$ consists of polynomials. Let $\tilde{\nu} \coloneqq \nu - m \cdot (s_1 + \cdots + s_\ell)$ so that $\CC(s,\nu) = \CC(s,\tilde{\nu}).$ The $\CC(s,\nu)$-vector space $V_{s,\nu}$ \eqref{eq:Vs,nu} is equal to the $\CC(s,\tilde{\nu})$-vector space $V_{s,\tilde{\nu}}$ defined by $\tilde{f}.$ 

The key tool to determine the dimension of the vector space $V_{s,\nu}$ is the {\em Mellin transform}. This allows to connect the integrals $I_{a,b}$ with the language of differential and shift operators. 
\begin{definition}\label{def:Mel} Let $f\in \CC[x_1,\ldots, x_n]^\ell$ be a tuple of polynomials and fix $s \in \CC^\ell.$ We define the {\em Mellin transform} of $f^s$ to be the function in the variables $\nu = (\nu_1,\dots,\nu_n)$ given by
\begin{equation} \label{eq:mellinmulti}
\MM \{f^s\}(\nu) \,\coloneqq\,  \int_{\Gamma} f^s\, x^\nu\, \frac{\d x}{x} \, = \, I_{0,0}(s,\nu),
\end{equation}
where $\Gamma \coloneqq \Gamma(s,\nu) \in H_n(X,\omega(s,\nu))$ is a twisted cycle, as defined in \Cref{sec:deRham}.
\end{definition}
The operator $\MM $ is naturally extended to functions $f^{s+a}x^b,$ i.e., $\MM (f^{s+a}x^b) = I_{a,b}(s,\nu).$

\begin{lemma} The Mellin transform obeys the following rules:
\begin{align}\label{eq:mellinrules}
\begin{split}
\MM \{x_i \cdot f^s\}(\nu) &\,=\, \MM \{f^s\}(\nu+e_i),\\
\MM \biggl\{ x_i\cdot  \frac{\partial f^s}{ \partial x_i}\biggr\}(\nu) &\,=\, -\nu_i \cdot \,  \MM \{f^s\}(\nu),
\end{split}
\end{align}
where again $e_i\in \RR^n$ denotes the $i$-th standard unit vector.
\end{lemma}
\begin{proof}
The first equality follows immediately from the definition. For the second, we again use the notation $\d x_{\widehat{i}}\in \Omega^{n-1}(X)$ introduced in \eqref{eq:dxk} and write out 
\begin{align*}
\MM \biggl\{ x_i\cdot  \frac{\partial f^s}{ \partial x_i}\biggr\}(\nu) & \,=\, \int_{\Gamma}\frac{\partial f^s}{\partial x_i}x^{\nu + e_i}\frac{\d x}{x} \,=\, \int_{\Gamma} \sum_{j=1}^n s_j \frac{1}{f_j}\frac{\partial f_j}{\partial x_i} f^s x^{\nu +e_i}\frac{\d x}{x}\\
& \,=\,(-1)^{i-1} \int_{\Gamma} \d \Big(f^s x^\nu \frac{\d x_{\widehat{i}}}{x_{\widehat{i}}}\Big)-\nu_i \cdot \MM (f^s)(\nu),
\end{align*}
where the last equality follows from Leibniz' rule. An explicit computation shows that $ (-1)^{i-1}\d (f^s x^\nu \frac{\d x_{\widehat{i}}}{x_{\widehat{i}}}) = f^s x^\nu \nabla_{\omega}(\frac{\d x_{\widehat{i}}}{x_{\widehat{i}}})$, hence
$$(-1)^{i-1} \cdot \int_{\Gamma} \d \Big( f^s x^\nu \frac{\d x_{\widehat{i}}}{x_{\widehat{i}}}\Big) \,=\, \langle[\Gamma],\Big[\nabla_{\omega}\Big(\frac{\d x_{\widehat{i}}}{x_{\widehat{i}}}\Big) \Big]\rangle \,=\, 0$$
by the perfect pairing introduced in \eqref{eq:perfectpairing}. This proves the second equality in \eqref{eq:mellinrules}.
\end{proof}
Therefore, the Mellin transform turns multiplication by $x_i^{\pm 1}$ into shifting the new variable $\nu_i$ by $\pm 1$ and the action of the $i$th Euler operator $\theta_i\coloneqq x_i\partial_i$ into multiplication by $-\nu_i.$ 

The techniques we are going to use to study the vector space $V_{s,\nu}$ come from $D$-module and Bernstein--Sato theory. We start with recalling basic definitions. 
For more details, we refer our readers to the references \cite{HTT08, SST00, SatStu19}.
The {\em $n$-th Weyl algebra}, denoted $$D_n \, \coloneqq \, \CC [x_1,\ldots,x_n]\langle \partial_1,\ldots,\partial_n \rangle,$$ or just $D$ if the number of variables is clear from the context, is a non-commutative ring gathering linear differential operators with polynomial coefficients. It is the free $\CC$-algebra generated by $x_1,\ldots,x_n$ and $\partial_1,\ldots,\partial_n $ modulo the following relations: all generators are assumed to commute, except $\partial_i$ and $x_i.$ Their commutator is $[\partial_i,x_i]\coloneqq \partial_i x_i - x_i \partial_i =1 \neq 0,$ where $i=1,\ldots,n.$ This reflects Leibniz' rule for taking the derivative of a product of functions. We denote the application of a differential operator to a function by the symbol~$\bullet.$ In this notation, $\partial_i \cdot x_i = x_i\cdot \partial_i +1 \in D_n, $ whereas $\partial_i \bullet x_i = 1 \in \CC[x_1,\ldots,x_n].$

In this paper, we study {\em left} $D$-ideals and $D$-modules, unless stated otherwise. Those ideals encode systems of linear PDEs in algebraic terms. A function that is annihilated by all operators contained in a $D$-ideal $I$ is called a {\em solution of $I$}. We use the notation $\theta_i \coloneqq x_i \partial_i$ for the $i$th Euler operator and $\theta$ for the vector $(\theta_1,\ldots,\theta_n)^\top \in D^n.$ 

We will also need the ring of global linear differential operators
\begin{align}
D_{\mathbb{G}_m^n} \,\coloneqq \, \CC [x_1^{\pm 1}, \ldots, x_n^{\pm 1}]\langle \partial_1, \ldots, \partial_n \rangle \,=\, \CC [x_1^{\pm 1}, \ldots, x_n^{\pm 1}] \langle \theta_1,\ldots,\theta_n \rangle
\end{align}
on the algebraic $n$-torus 
$\mathbb{G}_m^n = {\rm Spec}(\CC[x_1^{\pm 1},\ldots, x_n^{\pm 1}])$. 
In the rest of the document, we have been less strict about notation and used $(\CC^*)^n$ both for the algebraic $n$-torus $\mathbb{G}_m^n$ and its analytification $(\CC^*)^n,$ since it is clear from the context which one is meant. We here prefer to stick to the more careful distinction, which is also the standard in $D$-module theory.

In the study of the Mellin transform, we will use algebras of {\em shift operators}---also commonly called {\em difference operators}---with polynomial coefficients.
\begin{definition}\label{def:shiftalg}
The ($n$-th) {\em shift algebra} with polynomial coefficients
$$S_n \, \coloneqq\, \CC[\nu_1,\ldots, \nu_n]\langle  \sigma_1^{\pm 1},\ldots,\sigma_n^{\pm 1}\rangle $$  is the free \mbox{$\CC $-algebra} generated by $\{\nu_i,\sigma_i,\sigma_i^{-1}\}_{i=1,\ldots,n}$ modulo the following relations: all generators commute, except the variable $\nu_i$ and the {shift operator} $\sigma_i.$ They obey the~rule
\begin{align}\label{eq:commshift}
\sigma_i^{\pm 1}\nu_i \,=\, (\nu_i\pm 1)\sigma_i^{\pm 1}.
\end{align}
\end{definition}
This implies that $\sigma^a\nu^b = (\nu+a)^b\sigma^a$ for any $a\in\ZZ^n,$ $b\in\NN^n.$  The shift algebra naturally comes into play when studying the Mellin transform of functions. There is a natural action of $S_n$ on the Mellin transform of functions: it shifts the variable $\nu_i$ by $1$, i.e., $\sigma_i \bullet \MM\{f\}(\nu)=\MM\{f\}(\nu+e_i)$, which justifies the name ``shift operator'' and also explains the rule in~\eqref{eq:commshift}.
Mimicking the rules in \Cref{eq:mellinrules}, the {\em (algebraic) Mellin transform} (cf. \cite{loeser1991equations}) is the isomorphism of $\CC$-algebras 
\begin{equation} \label{eq:mellindiffop}
\MM \{\cdot\} \colon D_{\mathbb{G}_m^n} \longrightarrow S_n, \quad x_i^{\pm 1} \mapsto \sigma_i^{\pm 1}, \ \theta_i \mapsto -\nu_i.
\end{equation}
We conventionally use the notation $\MM \{ \cdot \}$ both for the Mellin transform of \emph{functions} (Definition \ref{def:Mel}) and that of \emph{operators} (Equation \eqref{eq:mellindiffop}). Note that $\MM \{\cdot \}$ naturally extends to an isomorphism of $D_{\mathbb{G}_m^n}[s_1,\ldots,s_\ell]$ and $S_n[s_1,\ldots,s_\ell]$ by mapping $s_j$ to itself.
\begin{remark}
Via the map $\MM\{ \cdot \},$ one can also define the Mellin transform $\MM\{ M \}$ of a $D_{\mathbb{G}_m^n}$-module $M.$ It is the following module over $S_n.$ As abelian groups, $\MM \{ M \}=M,$ and $S_n$ acts by $\nu_i \bullet m \coloneqq - \theta_i \bullet m$ and $\sigma_i^{\pm 1}\bullet m \coloneqq  x_i^{\pm 1}\bullet m$ for $m\in M,$ $i=1,\ldots,n.$
\end{remark}
We now recall Bernstein--Sato polynomials and ideals. We begin with the case $\ell = 1.$ 
\begin{definition}
The {\em Bernstein--Sato polynomial} $b_f\in \CC[s]$ of a polynomial $f\in \CC[x_1,\ldots,x_n]$ is the unique monic polynomial of smallest degree for which there exists $P_f\in D_n[s]$ such~that 
\begin{equation}\label{eq:BS}
P_f(s)\bullet f^{s+1} \,=\, b_f(s) \cdot f^s.
\end{equation}
\end{definition}
If $f$ is smooth, $b_f=s+1.$
Observe that while the Bernstein polynomial is unique, the Bernstein--Sato operator $P_f$ is unique only modulo $\Ann_{D_n[s]}(f^{s+1}).$ It is known that $b_f$ is non-trivial and that its roots are negative rational numbers~\cite{Kas76}.
Bernstein--Sato polynomials were originally studied to construct a meromorphic continuation of the distribution-valued function $s \mapsto f^s,$ which is a priori defined only for complex numbers $s\in \CC$ with positive real part. Nowadays, it is an important object of study in singularity theory, among others in work on the monodromy conjecture such as \cite{BVW21,bvwz2}. 

Applying the Mellin transform to both sides of \Cref{eq:BS} yields
\begin{equation}\label{eq:lowering}
b_f(s)\cdot \MM \{f^s\} \,\, = \,\, \MM \{ P_f \}\bullet \MM \{f^{s+1}\}.
\end{equation}
We refer to this relation as being {\em lowering in $s$}. This means that it provides a way for writing the integral $I_{0,0}(s,\nu)$ as a linear combination of integrals of type $I_{0,b}(s+1,\nu)$ for some $b\in\ZZ^n.$
One obtains a {\em raising} relation by the simple trick of considering $f\in \CC[x]$ as a differential operator of order~$0$: 
\begin{equation}\label{eq:raising}
\MM \{f^{s+1}\} \,\, = \,\, \MM \{f\cdot f^s\} \,\,=\,\, \underbrace{\MM \{f\}}_{\in S_n}\bullet \, \MM \{f^s\}.
\end{equation}

\begin{definition}
The {\em $s$-parametric annihilator of $f^s$} is the $D_n[s]$-ideal $$ \Ann_{D_n[s]}(f^s) \, \coloneqq \, \left\{ P \in D_n[s] \mid P \bullet f^s =0 \right\} .$$
\end{definition}
For $P$ as in \Cref{eq:BS}, the operator $P_ff-b_f$ is in $ \Ann_{D_n[s]}(f^s).$ Moreover, for any operator $P\in \Ann_{D_n[s]}(f^s),$ applying the Mellin transform to the equation $P\bullet f^s=0,$ 
one attains a $\CC(s,\nu)$-linear relation among integrals $I_{a,b}.$
The following example shows how to compute this ideal in practice.
\begin{example}
Let $f=(x-1)(x-2)\in\CC[x].$ The $s$-parametric annihilator of $f^s$ can be computed running the following code in the computer algebra system~{\sc Singular} \cite{DGPS,Plural} using the library {\tt dmod\_lib}~\cite{ABLMS}:
\begin{small}
\begin{verbatim}
LIB "dmod.lib";
ring r = 0,x,dp; poly f = (x-1)*(x-2);
def A = operatorBM(f); setring A;
LD; // s-parametric annihilator of f^s
\end{verbatim}\end{small}
In this case, the $D_n[s]$-ideal $\Ann_{D_n[s]}(f^s)$ is generated by the operator \mbox{$P = f\partial_x-s\partial_x\bullet f.$}
A linear relation in $V_{s,\nu}$ among integrals of the form $I_{a,b}$ can be attained by expanding the equation $\MM \{P\}\bullet \MM \{f^s\}\,=\,0.$ The same relation is given with the method as in \Cref{ex:elliptic0} by taking the $0$-form $\phi = f.$    
\end{example}
\begin{proposition}\label{prop:b0}
Fix $f\in \CC[x_1,\ldots,x_n]$ and $\Gamma\in H_n(X,\omega).$ The following $\CC(s,\nu)$-vector spaces coincide: 
$$ V_{s,\nu}\,\,=\,\, \sum_{a \,\in\, \ZZ^{\ell}} \CC(s,\nu) \otimes_{\CC[s,\nu]} \left(S_n \bullet I_{a,0}\right)\,\,=\,\, \CC(s,\nu) \otimes_{\CC[s,\nu]} \left(S_n[s]\bullet I_{0,0}\right).$$ 
\end{proposition}
\begin{proof}
The first equality follows from $\sigma^b\bullet I_{a,0}=I_{a,b}.$ The second  equality follows from the fact that via the Mellin transform, one obtains both lowering and increasing shift relations in $s$ as in Equations \eqref{eq:lowering} and \eqref{eq:raising}.
\end{proof}
This statement is contained in \cite{FeynmanAnnihilator} for the special case that $\Gamma = \RR_{>0}^n$ and $f$ is a polynomial. 
Elements of a basis of $V_{s,\nu}$ are called {\em master integrals} in physics literature. 
We now present the main result of \cite{FeynmanAnnihilator}, which relates the number of master integrals to the topological Euler characteristic of our very affine variety for the case $\ell =1.$

\begin{theorem}[{\cite[Corollary 37]{FeynmanAnnihilator}}] \label{thm:dimVsnu}
The dimension of $V_{s,\nu}$ is given be the signed topological Euler characteristic of the hypersurface complement $(\CC^{\ast})^n\setminus V(f),$ i.e.,
$$\dim_{\CC(s,\nu)}(V_{s,\nu}) \,=\,
(-1)^n \cdot \chi \left( \left( \CC^{\ast} \right)^n\setminus V(f)\right).$$
\end{theorem} 
The proof of this statement in \cite[Section 3]{FeynmanAnnihilator} builds on work of Loeser and Sabbah \cite{LoeSab,LoeSab2}. 
Interested readers can find more details in the proof of \Cref{thm:dimVsnumulti}, in which we  generalize \Cref{thm:dimVsnu} to arbitrary $\ell>1.$

The following example illustrates how to obtain shift relations among integrals when ${f\in \CC[x_1\dots,x_n]}$ is smooth, starting from a Bernstein--Sato operator. We also exhibit the $(n-1)$-form to which the very same relation corresponds via the method presented in~\Cref{sec:deRham}.
\begin{example}
Let $f\in \CC[x_1,\ldots,x_n]$ be smooth. Hence $b_f=(s+1)\in \CC [s].$ Since $f$ and its partial derivatives are coprime, there exist polynomials $p_1,\ldots,p_n,q\in \CC [x]$ such that $(\sum_{k=1}^n p_k\partial_k\bullet f)+qf=1.$ Then $P_f =( \sum_k p_k\partial_k) + b_fq\in D_n[s]$ is a Bernstein--Sato operator of $f$ since
\begin{align*} 
P_f\bullet f^{s+1} &= (s+1)f^s\cdot \sum_{k=1}^n p_k\frac{\partial f}{\partial x_k} + (s+1)q\cdot f^{s+1} \\
& = (s+1) \cdot \underbrace{(\sum_{k=1}^n p_k\frac{\partial f}{\partial x_k}+qf)}_{=1}\cdot f^s = b_f\cdot f^s .
\end{align*}
Hence $P_ff-b_f \in \Ann_{D_n[s]}(f^s).$ Given a polynomial $p\in \CC[x],$ we denote by $p( \sigma )$ its image under the isomorphism~\eqref{eq:mellindiffop}. The Mellin transform of $P_f$ is
$$\MM \{P_f\} \,\,=\,\,  \sum_{k=1}^n\left((1-\nu_k) p_k(\sigma) \sigma_k^{-1}-\frac{\partial p_k}{\partial x_k}(\sigma)\right)+(s+1)q(\sigma) \,\in\, S_n[s]. $$
Applying the Mellin transform to \Cref{eq:BS} induces the shift relation
\begin{align*}
\sum_{k=1}^n\left((1-\nu_k) p_k(\sigma) \sigma_k^{-1}-\frac{\partial p_k}{\partial x_k}(\sigma)\right)\bullet \MM\{f^{s+1}\} + (s+1)q(\sigma)\bullet\MM\{f^{s+1}\}\,=\,(s+1) \MM\{f^{s}\}.
\end{align*}
Expanding this equality, one obtains a $\CC(s,\nu)$-linear combination of integrals of type \eqref{eq:Vs,nu}. Precisely the same relation among integrals is attained with the method illustrated in~\eqref{eq:relInCohomology} for the $(n-1)$-form $\phi = f\cdot \sum_{k=1}^n p_k \frac{\, \d x_{\widehat{k}}}{x}\in\Omega^{n-1}_X(X).$ The image under $\nabla_\omega$ of a single summand of $\phi$ is displayed in \Cref{eq:imageNabla}. The correspondence between annihilating operators and $(n-1)$-forms is stated more clearly in the next result.  
\end{example}

\begin{proposition} \label{prop:3.11}
Let $\ell = 1$ and consider a differential operator $P\in \Ann_{D_n[s]}(f^s)$ which is of degree at most $1$ in $\partial_1,\dots,\partial_n,$ i.e., 
$$P \,=\,\sum_{i=1}^n p_i(x,s) \cdot \partial_i + q(x,s),\quad  \hbox{where}\,\, p_1,\dots,p_n,q \,\in\, \CC[x_1,\dots,x_n,s].$$
Then the equalities  $\MM \{P\}\bullet\MM \{ f^s\}=0$ 
and \eqref{eq:nablaPhi} with $\phi = \sum_{i=1}^n(-1)^{i-1} p_i  \frac{\d x_{\hat{i}}}{x}$ lead to the same linear relation of integrals $I_{a,b}.$ 
\end{proposition}
\begin{proof}
Since $P\in \Ann_{D_n[s]}(f^s),$ we can determine the polynomial $q$ in terms of the $p_i$:
$$\left(\sum_{i=1}^n s_i p_i f^{s-1}\partial_i \bullet f\right)f^s + qf^s \,=\,0 \,\,\iff\,\, q\, =\, - \sum_{i=1}^n s_i p_i f^{s-1}\partial_i \bullet f.$$
An explicit computation shows that the relation obtained from $\MM \{P\bullet f^s\}=0,$ where
\begin{equation} \label{eq:mellinrelation}
\MM \{P\} \,=\, - \left(\sum_{i=1}^n(\nu_i-1)p_i(\sigma)\sigma_i^{-1} + \sum_{i=1}^n\frac{\partial p_i}{\partial x_i}(\sigma) + s \frac{1}{f(\sigma)}\sum_{i+1}^n \frac{\partial f}{\partial x_i}(\sigma) p_i(\sigma)\right),
\end{equation}
coincides with the one in cohomology obtained from \eqref{eq:nablaPhi}. This is immediate when computing explicitly 
$$\nabla_{\omega}\left(\sum_{i=1}^n(-1)^{i-1}p_i  \frac{\d x_{\widehat{i}}}{x} \right) \,=\, \left(\sum_{i=1}^n(\nu_i-1)p_i x_i^{-1} + \sum_{i=1}^n\frac{\partial p_i}{\partial x_i} + s \frac{1}{f}\sum_{i=1}^n \frac{\partial f}{\partial x_i} p_i\right) \frac{\d x}{x},$$
where $\nabla_\omega(\phi)=\d \phi + \omega \wedge \phi,$ and $\omega$ is our one-form from \eqref{eq:omega}. To be precise, in this last computation, we take $s = s^*$ and $\nu = \nu^*$ to be fixed, generic complex numbers. The coefficients $C_{a,b}(\phi)$ in the relation \eqref{eq:nablaPhi} are obtained from evaluating the rational function coefficients of $I_{a,b}$ in $\MM \{P \} \bullet \MM \{f^s\} = 0$ at $(s,\nu) = (s^*,\nu^*)$.  
\end{proof}

For $\ell>1$ polynomials $f_1,\ldots,f_{\ell},$ one needs to study Bernstein--Sato {\em ideals} instead.
The {\em Bernstein--Sato ideal} of  $(f_1,\ldots,f_{\ell})$ is the ideal $B_{(f_1,\ldots,f_{\ell})} \triangleleft \CC[s_1,\ldots,s_\ell]$ consisting of all polynomials $p\in \CC[s_1,\ldots,s_\ell]$ for which there exists $P\in D_n[s_1,\ldots,s_\ell]$ s.t. 
\begin{align}\label{eq:BSideal}
P\bullet \left( f_1^{s_1+1}\cdots f_\ell^{s_\ell +1} \right) \,\, = \,\, p \cdot f_1^{s_1}\cdots f_\ell^{s_\ell } .
\end{align}
Sabbah \cite{Sab87} proved that $B_F$ is non-trivial and moreover that the irreducible components of  $V(B_F)$ of codimension one are affine-linear hyperplanes with nonnegative rational coefficients. This is analogous to the fact that the zeroes of the Bernstein--Sato polynomial are negative rational~numbers.

In~\eqref{eq:BSideal}, no individual shifts in the variables $s_i$ can be taken into account; only a simultaneous shift by the all-one vector. 
A remedy is provided by the following ideals of Bernstein--Sato type, which also enter the study of the monodromy conjecture in \cite{BVW21,bvwz2}. 
\begin{definition}\label{def:aBS}
Let $a\in \NN^\ell$ be a non-negative integer vector. The {\em $a$-Bernstein--Sato ideal} of $(f_1,\ldots,f_\ell)$ is the ideal $B_{(f_1,\ldots,f_{\ell})}^a \triangleleft \, \CC[s_1,\ldots,s_\ell]$ consisting of all polynomials $p\in \CC[s_1,\ldots,s_\ell]$ for which there exists $P\in D_n[s_1,\ldots,s_\ell]$ such that
\begin{align}\label{eq:aBS}
P\bullet \left( f_1^{s_1+a_1}\cdots f_\ell^{s_\ell +a_\ell}\right) \,\,=\,\, \,\, p \cdot f_1^{s_1}\cdots f_\ell^{s_\ell } . 
\end{align}
\end{definition}
Again by~\cite{Sab87}, the $a$-Bernstein--Sato ideal is non-trivial.
In this notation, for $a=(1,\ldots,1)$ the all-one vector,  $B_{(f_1,\ldots,f_{\ell})}^{(1,\ldots,1)}=B_{(f_1,\ldots,f_{\ell})}.$
At present, the computation of $a$-Bernstein--Sato ideals using computer algebra software is out of reach.
Yet, we can use $a$-Bernstein--Sato ideals to generalize \Cref{prop:b0} to a tuple of $\ell$ polynomials as follows.
 
\begin{proposition}\label{prop:b0multi}
Fix $\Gamma \in H_n(X,\omega)$ and let $f_1,\ldots,f_{\ell}\in \CC[x_1,\ldots,x_n].$ The following $\CC(s,\nu)=\CC(s_1,\ldots, s_{\ell},\nu_1,\ldots,\nu_n)$-vector spaces coincide: 
\begin{align}
\begin{split}
V_{s,\nu} &\,=\, 
\sum_{(a,b) \,\in \, \ZZ^{\ell} \!\times \!\,\ZZ^{n}} \CC(s,\nu)\cdot  I_{a,b}  \,\,=\,\, \sum_{a\in \ZZ^{\ell}} \CC(s,\nu) \otimes_{\CC[s,\nu]} \left( S_n \bullet I_{a,0}\right) \\
&\,\,=\,\, \CC(s,\nu) \otimes_{\CC[s,\nu]} \left( S_n[s_1,\ldots,s_{\ell}]\bullet I_{0,0}\right)  \, .
\end{split}
\end{align}
\end{proposition}
\begin{proof}
Let $e_i\in \NN^{\ell}$ denote the $i$-th standard unit vector. Then 
$$ \MM \{f^{s+e_i}\} \,=\, \MM \{f_i \cdot f^s\} \,=\, \MM \{f_i\}\bullet  \MM \{f^s\} $$
is an increasing relation in $s_i.$
To construct lowering relations in the $s_i,$ we study $a$-Bernstein--Sato ideals for $a=e_i.$ Applying the Mellin transform to both sides of \Cref{eq:aBS} yields
$$  \MM \{P\} \bullet \MM  \{f^{s+e_i} \}  \,=\, p \cdot  \MM \{f^s\} ,$$
a lowering shift relation in $s_i.$
\end{proof}

With those tools at hand, we can now generalize \Cref{thm:dimVsnu} to $\ell>1.$  
Denote by $D_n(s)\coloneqq D_n[s]\otimes_{\CC[s]}\CC(s)$ the $n$-th Weyl algebra over the field $\CC(s)=\CC(s_1,\ldots,s_{\ell}).$
\begin{theorem}\label{thm:dimVsnumulti} 
Let $f\in \CC[x_1,\ldots,x_n]^{\ell}$ and denote by $\cM $ the  $D_n(s)$-module $D_n(s)\bullet f^s.$ Then 
\begin{align*}
\dim_{\CC (s,\nu)} \left( V_{s,\nu} \right) \, =\, (-1)^n\cdot\chi\left(\iota^{\ast} \cM  \right)
\, =\, {(-1)^n\cdot}\chi\left((\CC^{\ast})^n\setminus V(f_1\cdots f_\ell) \right),
\end{align*}
where $\iota\colon \mathbb{G}_{\CC(s),m}^n \hookrightarrow \mathbb{A}_{\CC(s)}^n$ denotes the embedding of the algebraic $n$-torus over $\CC(s)$ into the affine $n$-space over $\CC(s),$ and $\iota^* \cM $ is the $D$-module pullback of $\cM $ via $\iota.$
 \end{theorem}
\begin{proof} We first prove that $\iota^* \cM =\cM [x^{-1}]=D_{\mathbb{G}_m^n(s)}\bullet f^s $ is holonomic.
Note that the operator 
\begin{align*} 
P_i(s) \,\coloneqq \,f_1\cdots f_{\ell}\cdot \partial_i - \sum_{j=1}^{\ell} s_j f_{\widehat{j}} \, \frac{\partial f_j}{\partial x_i} \, \in  \,\Ann_{D_n[s_1,\ldots,s_{\ell}]}\left( f^s \right)
\end{align*}
annihilates $f^s=f_1^{s_1}\cdots f_{\ell}^{s_{\ell}}$ for all $i=1,\ldots,n,$  where $f_{\widehat{j}}$ denotes $f_1\cdots f_{j-1}\cdot f_{j+1}\cdots f_{\ell}.$
Denote by $I$ the $D_n[s]$-ideal generated by $P_1(s),\ldots,P_n(s).$ Clearly, $I\subseteq \Ann_{D_n[s]}(f^s).$
Hence, for every $s^*\in \CC^{\ell},$ the $D_n$-ideal $I_{s^*}\coloneqq \langle P_1(s^*),\ldots,P_n(s^*)\rangle $ has finite holonomic rank and therefore its Weyl closure $W(I_{s^*})=R_nI_{s^*} \cap D_n$ is holonomic by \cite[Theorem 1.4.15]{SST00}. Since $I_{s^*}\subseteq W(I_{s^*})\subseteq \Ann_{D_n}(f^{s^*}),$ this proves that for all $s^*\in \CC^{\ell},$ the $D_n$-ideal $\Ann_{D_n}(f^{s^*})$ is holonomic. 
Hence $ D_n(s) \bullet f^s \,\cong\, D_n(s)/\Ann_{D_n(s)}(f^s)$
is a holonomic $D_n(s$)-module. 
Now denote by ${D_{\mathbb{G}_{m}^n}(s)=D_{\mathbb{G}_m^n}[s]\otimes_{\CC[s]}\CC(s)}$ the Weyl algebra over the algebraic $n$-torus over $\CC(s).$
By \cite[Theorem 3.2.3]{HTT08}, also the $D_{\mathbb{G}_{m}^n}(s)$-module $\iota^* \cM =\cM [x^{-1}]=D_{\mathbb{G}_m^n}(s)/\Ann_{D_{\mathbb{G}_m^n}(s)}(f^s)$ is holonomic. Therefore, by a classical result of Kashiwara, its solution complex---or, equivalently, its de Rham complex---is an element of the bounded derived category of constructible sheaves. Therefore, $\iota^* \cM $ has finite Euler characteristic by Kashiwara's index theorem for constructible sheaves~\cite{Kas85}.

For the rest of the proof, we follow and adapt the strategy of proof of \cite[Section 3]{FeynmanAnnihilator} to the case~$\ell > 1.$ The next step is to show that $\dim_{\CC(s,\nu)}(V_{s,\nu})$ equals the Euler characteristic of the de Rham complex of $\iota^*\cM .$
Recall that $\Ann_{D_{\mathbb{G}_m^n[s]}}(f^s)$ turns into $\Ann_{S_n[s]}(I_{0,0})$ via the Mellin transform \eqref{eq:mellindiffop}. By \Cref{prop:b0multi}, we hence obtain the equality
\begin{align*}
\begin{split}
\dim_{\CC(s,\nu)} \left (V_{s,\nu}\right) &\, = \, \dim_{\CC(s,\nu)}\left (\CC (s,\nu) \otimes_{\CC[s,\nu]} \left( S_n[s] \bullet I_{0,0} \right) \right) \\ 
& \,=\, \dim_{\CC(s,\theta)}\left( \CC(s,\theta) \otimes_{\CC[s,\theta]} \left( D_{\mathbb{G}_m^n}[s] \bullet f^s \right)\right),
\end{split}
\end{align*}
where $\theta=(\theta_1, \ldots ,\theta_n).$ By \cite[Th\'{e}or\`{e}me 2]{LoeSab2} and noting that  we work over the torus over~$\CC(s)$, \begin{align*}
\dim_{\CC (s,\theta )} \left( \CC (s,\theta) \otimes_{\CC[s,\theta]} \left( D_{\mathbb{G}_m^n}[s] \bullet f^s\right) \right) \,=\, (-1)^n\cdot\chi \left( \iota^*  \cM  \right) .
\end{align*}
Now note that each element of $D_{\mathbb{G}_m^n}(s) \bullet f^s $ can be uniquely written as $h\cdot f^s$ for some $h\in \CC(s)[x^{\pm 1},f_1^{-1},\ldots,f_{\ell}^{-1}].$ 
The natural action of  $D_{\mathbb{G}_m^n}(s)$ on $\iota^*\cM $ is given by
\begin{align*}
\partial_i \bullet (hf^s) \, = \, \frac{\partial h}{\partial x_i}f^s + \sum_{j=1}^{\ell} h s_j \frac{1}{f_j} \frac{\partial f_j}{\partial x_i}f^s.
\end{align*}
Moreover, the morphism 
\begin{align*}
    D_{\mathbb{G}_m^n}(s) \bullet f^s  \longrightarrow \CC(s) \left[ x^{\pm 1},f_1^{-1},\ldots,f_\ell^{-1}\right], \quad  h\cdot f^s\mapsto h
\end{align*}
is an isomorphism of $\CC(s)[x^{\pm 1}]$-modules. 
It remains to prove that $\chi(\iota^* \cM)$ is equal to the signed Euler characteristic of $(\CC^*)^n\setminus V(f_1\cdots f_{\ell}).$ 
For that, we first prove that the $D_{\mathbb{G}_m^n}(s)$-module $\iota^* \cM$ and the $D_n$-module $\CC[x^{\pm 1},f_1^{-1},\ldots,f_{\ell}^{-1}]$ have the same Euler characteristic.

Denote by $\cN $ the $D_n$-module $\CC[x^{\pm1 }]$ and by $\cN [f^{-1}]= \cN[f_1^{-1},\ldots,f_{\ell}^{-1}].$ Then $D_{\mathbb{G}_m^n }\bullet f^s$  and $\cN[f^{-1}][s]$ are isomorphic as $\CC$-vector spaces. We now introduce new variables $t_1,\ldots,t_{\ell}$ that commute with $x_1,\ldots,x_n$ and consider $D_{\mathbb{G}_m^n }\bullet f^s$ as module over $D_{\mathbb{G}_m^{n+\ell}}=\CC[x^{\pm 1},t^{\pm 1}]\langle \partial_x,\partial_t\rangle$~via
\begin{align}\label{eq:actiont}
t_i^{\pm 1} \bullet \left(n(s)f^s\right)\,\coloneqq \, n(s\pm e_i)f^{s}f_i^{\pm 1} \quad \text{and} \quad \partial_{t_i}\bullet \left( n(s)f^s \right)\,\coloneqq \, -s_i\cdot n(s-e_i)f^s\frac{1}{f_i}.
\end{align}
Since $\partial_{t_i}t_i=-s_i,$ the following $D_{\mathbb{G}_m^{n+\ell}}$-modules are isomorphic by the rules in \eqref{eq:actiont}:
\begin{align*}
\begin{split}
& \hspace*{-7mm} \left(\left( \left(D_{\mathbb{G}_m^n}\bullet f^s / \partial_{t_{\ell}}\bullet \left( D_{\mathbb{G}_m^n}\bullet f^s\right) \right) / \partial_{t_{\ell -1}} \bullet \left[ D_{\mathbb{G}_m^n}\bullet f^s \right]\right) /  \cdots \right)/\partial_{t_1}\bullet \left[ D_{\mathbb{G}_m^n}\bullet f^s \right]  \\
   & \qquad  \qquad \qquad =\,  \frac{D_{\mathbb{G}_m^n }\bullet f^s}{\partial_{t_1}\bullet\left( D_{\mathbb{G}_m^n }\bullet f^s \right)+ \cdots + \partial_{t_{\ell}} \bullet \left(D_{\mathbb{G}_m^n }\bullet f^s \right)} \\
    & \qquad  \qquad \qquad =\, \frac{D_{\mathbb{G}_m^n }\bullet f^s}{s_1t_1^{-1}\bullet\left( D_{\mathbb{G}_m^n }\bullet f^s \right) + \cdots + s_{\ell}t_{\ell}^{-1}\bullet \left(D_{\mathbb{G}_m^n }\bullet f^s\right) }\\ 
    & \qquad  \qquad \qquad =\,\frac{D_{\mathbb{G}_m^n }\bullet f^s}{s_1 \cdot D_{\mathbb{G}_m^n }\bullet f^s + \cdots +s_{\ell}\cdot D_{\mathbb{G}_m^n }\bullet f^s } \,\, \cong \,\, \cN[f^{-1}] \, .
\end{split}
\end{align*}
Since $\partial_{t_{\ell}}$ is injective on $D_{\mathbb{G}_m^n }\bullet f^s,$ the proof of \cite[Theorem 35]{FeynmanAnnihilator} shows that
$\chi(D_{\mathbb{G}_m^n }\bullet f^s) = \chi (D_{\mathbb{G}_m^n}\bullet f^s / \partial_{t_{\ell}}\bullet ( D_{\mathbb{G}_m^n}\bullet f^s ))$. 
By iterating this reasoning, we conclude~that
\begin{align}\label{eq:chiDfs}
    \chi\left( D_{\mathbb{G}_m^n }\bullet f^s \right) \, = \, \chi\left( \cN[f^{-1}] \right).
\end{align}
Now denote by $\MM ^t\{ D_{\mathbb{G}_m^n }\bullet f^s\}$ the Mellin transform of $D_{\mathbb{G}_m^n }\bullet f^s$ with respect to the variables $t_1,\ldots,t_{\ell}.$ Then $\MM ^t\{ D_{\mathbb{G}_m^n }\bullet f^s \}(s)\cong (D_{\mathbb{G}_m^n }\bullet f^s)({t_1}\partial_{t_1},\ldots,t_{\ell}\partial_{t_{\ell}}),$ since tensoring with $\CC(t\partial_t)$ just extends the coefficients to $\CC(s).$ 
Again by~\cite{LoeSab2}, 
\begin{align*}
\begin{split}{(-1)^n\cdot}\chi \left( D_{\mathbb{G}_m^n }\bullet f^s\right) & \,=\, \dim_{\CC(\theta,t\partial_t)} \left( \left(D_{\mathbb{G}_m^n }\bullet f^s\right)(\theta,t\partial_t)\right)   \,=\, \dim_{\CC(\theta,s)} \left( \left(D_{\mathbb{G}_m^n }\bullet f^s\right)(s,t\partial_t)\right)\\
& \, =\, \dim_{\CC(s,\theta)}  \left( \MM ^t(D_{\mathbb{G}_m^n }\bullet f^s)(s,\theta) \right)   \,=\,{(-1)^n\cdot} \chi\left( \MM ^t\left(D_{\mathbb{G}_m^n }\bullet f^s\right)(s) \right),
\end{split}
\end{align*}
see in particular \cite[Theorem 35]{FeynmanAnnihilator} for the first equality.
Therefore,
\begin{align*}
\MM ^t(D_{\mathbb{G}_m^n }\bullet f^s )(s) \, \cong \, \CC(s)[x^{\pm 1},f_1^{-1},\ldots,f_{\ell}^{-1}]\cdot f^s
\end{align*} 
is $\iota^* \cM.$
Hence, by \eqref{eq:chiDfs}, $\chi(\iota^*\cM) = \chi (\CC[x^{\pm 1},f_1^{-1},\ldots,f_{\ell}^{-1}]),$ concluding the proof.
\end{proof}

\begin{remark} Alternatively, one can prove~\Cref{thm:dimVsnumulti} by an inductive argument. We demonstrate how to reduce the proof of the statement from $\ell$ to $\ell -1$ polynomials.
Let $\cM$ be a $D_n$-module and $f_1,f_2\in \CC[x_1,\ldots,x_n]$ two polynomials. 
Consider the module $\cM _1\coloneqq \cM [f_1^{-1}]$. By applying \cite[(3.13)]{FeynmanAnnihilator} to $\cM _1,$ we see that $\chi(\cM _1[s_2,f_2^{-1}]f_2^{s_2}) = \chi(\cM _1[f_2^{-1}]).$
More precisely, one gets
$\chi(\cM [s_2,f_1^{-1},f_2^{-1}]f_2^{s_2}) = \chi(\mathcal{ M}[f_1^{-1},f_2^{-1}]).$
Now denote $\cM _2 \coloneqq \cM [s_2,f_2^{-1}]f_2^{s_2}.$ Again by \cite[(3.13)]{FeynmanAnnihilator}, we get that $ \chi(\cM _2[f_1^{-1}]) = \chi(\cM _2[s_1,f_1^{-1}]f_1^{s_1})$
and hence
$\chi(\cM [s_1,s_2,f_1^{-1},f_2^{-1}]f_1^{s_1}f_2^{s_2}) = \chi(\cM [f_1^{-1},f_2^{-1}]).$ Iterating this process and setting $\cM$ to be $\CC[x_1^{\pm 1},\ldots,x_n^{\pm 1}]$ concludes the reasoning.
\end{remark}  

\section{GKZ systems}\label{sec:GKZ}
It is well known that generalized Euler integrals provide a full description of the solutions to systems of linear PDEs called {\em GKZ systems} or {\em $A$-hypergeometric systems}~\cite{GKZ90}. Recent works by Matsubara-Heo and Takayama \cite{Matsubara2020,MatsubaraGKZ,MatsubaraHeo2021,MT21} expose connections with previous sections. We review some of these results and demonstrate how to compute with GKZ systems.

Throughout the section, we consider the parameters $\nu\in \CC^n,$ $ s\in \CC^{\ell},$ and the integers $a=b=0$ to be fixed. We view the integrals \eqref{eq:integrals_intro} as functions of the coefficients of the Laurent polynomials $f_j.$ Before making this precise, we introduce some additional notation.

We fix finite subsets $\{A_j\}_{ j=1,\ldots, \ell}$ of $\ZZ^n$ representing the monomial supports of the $f_j$:
\begin{align} f_j(x; c_j) \, = \, \sum_{\alpha \in A_j} c_{\alpha,j} \, x^\alpha. 
\end{align} 
The parameters $c_j = (c_{\alpha,j})_{\alpha \in A_j}$ take values in $\CC^{A_j} \coloneqq \CC^{|A_j|}.$ The {\em Cayley configuration} of $\{A_j\}_{1\leq j \leq \ell}$ is $\{ (\alpha, e_j) \,|\, \alpha \in A_j \},$ given by the columns of the $(n+\ell) \times \sum_{j=1}^\ell |A_j|$ matrix
\begin{equation} \label{eq:cayley}
A = \left( \begin{array}{ccc|ccc|c|ccc}
&A_1& & &A_2& &  & &A_\ell& \\
1 & \cdots & 1 & 0 & \cdots & 0 &  & 0 & \cdots & 0\\
0 & \cdots & 0 & 1 & \cdots & 1 & \cdots & 0 & \cdots & 0 \\
& \vdots & & & \vdots & & & & \vdots & \\
0 & \cdots & 0 & 0 & \cdots & 0 & & 1 & \cdots & 1
\end{array} \right). 
\end{equation}
Here, $A_j$ is represented by an $n \times |A_j|$ matrix, and $e_j$ is the $j$-th standard unit vector in~$\NN^\ell.$
It will be convenient to view $A$ as the disjoint union of $A_1, \ldots, A_\ell$ and to collect all coefficients in a vector $c = (c_{\alpha})_{\alpha \in A},$ with entries indexed by $A.$ The parameters $c$ take values in $\CC^A \coloneqq \CC^{A_1} \times \cdots \times \CC^{A_\ell} =  \CC^{|A_1| + \cdots + |A_\ell|}.$ 

The very affine variety $X$ from \eqref{eq:X_intro} now depends on the choice of coefficients. We write
\begin{align}\label{def:Xc}
X(c) \, \coloneqq \, (\CC^*)^n \setminus V\big( \prod_{j = 1}^\ell f_j\left(x;c_j\right)\big) .
\end{align}
For fixed $c^* \in \CC^A,$ let $\Gamma_{c^*} = \Delta \otimes_{\CC}\phi_{c^*}$ be a cycle in the vector space $C_n(X(c^*),{\cal L}_{\omega(c^*)}^\vee)$ from the twisted chain complex introduced in \Cref{sec:deRham}. For $c$ in a sufficiently small neighborhood $U_{c^*} \subset \CC^A$ of $c^*,$ the singular chain $\Delta$ is contained in $X(c)$ as well. The $1$-form $\omega(c)$ depends rationally on $c,$ and there is a unique section $\phi_c$ of ${\cal L}_{\omega(c)}^\vee$ such that $\phi_c$ is obtained from $\phi_{c^*}$ by analytic continuation. Varying $c$ in a small neighborhood $U_{c^*},$ we obtain a function 
\begin{equation} \label{eq:cfunction}
I_{\Gamma_{c^*}} \, \coloneqq \, \left( c \, \mapsto \, \int_{\Gamma_{c^* \rightarrow c}} f(x;c)^s \, x^\nu \, \frac{\d x}{x}\right),
\end{equation} 
with $f(x;c) = (f_j(x;c_j))_{1 \leq j \leq \ell}.$ The twisted cycle $\Gamma_{c^* \rightarrow c} = \Delta \otimes_{\CC}\phi_c$ over which is integrated depends on $c,$ as well as on $\Gamma_{c^*}.$ The vector space $V_{c^*}$ from the introduction is generated by the functions \eqref{eq:cfunction}, where $[\Gamma_{c^*}]$ ranges over the twisted homology vector space $H_n(X(c^*),\omega(c^*))$ and two such functions are identified if they coincide on an open subset containing $c^*.$

We will now write down differential operators which annihilate the functions \eqref{eq:cfunction}. We consider the  Weyl algebra \mbox{$D_A =\CC[c_\alpha\,|\,\alpha \in A] \langle \partial_a \,|\, \alpha \in A \rangle$} whose variables are indexed by the columns of $A.$
The {\em toric ideal} associated to $A$ is the binomial ideal 
\begin{equation}\label{eq:toricIdeal}
I_A \, \coloneqq\, \langle\partial^u-\partial^v \,|\, u-v \in \ker(A), \; u,v\in \NN^A\rangle \, \triangleleft \,\CC [\partial_\alpha \,|\, \alpha \in A].
\end{equation}
Here we use the notation $u = (u_\alpha)_{\alpha \in A} \in \NN^A$ and $\partial^u = \prod_{\alpha \in A} \partial_\alpha^{u_\alpha},$ and similarly for $v.$ For the convenience of the reader, we now verify that any operator in $I_A$ indeed annihilates $I_{\Gamma_{c^*}}.$ One checks that 
\begin{equation} \label{eq:diffrule}
\partial^u \bullet I_{\Gamma_{c^*}}(c) \, = \, \int_{\Gamma_{c^* \rightarrow c}} s (s-1)\cdots(s-|u|+1) \, f(x;c)^{s-|u|} \, x^{Au + \nu} \, \frac{\d x}{x}, 
\end{equation} 
where $|u| = \sum_{\alpha \in A} u_a.$ If $u - v \in \ker(A),$ we have $|u|=|v|,$ and $Au = Av,$ which proves $(\partial^u-\partial^v) \bullet I_{\Gamma_{c^*}} = 0.$ The ideal $I_A$ is called {\em toric} because, viewed as an ideal in the polynomial ring, it defines a toric variety. We will show how to compute its generators in \Cref{ex:gkz}. 

We now define another $D$-ideal of differential operators annihilating our integral functions~\eqref{eq:cfunction}. This ideal, in contrast to $I_A,$ will depend on the exponents $s$ and $\nu.$ We write $\kappa \in \CC^{n+\ell}$ for the vector $(-\nu, s)^\top.$ Let $J_{A,\kappa}$ be the ideal generated by the entries of $ A\theta-\kappa$ where $\theta \coloneqq  (\theta_\alpha)_{\alpha \in A}$ and $\theta_\alpha = c_\alpha \partial_\alpha.$
It is well known that $P \bullet I_{\Gamma_{c^*}} = 0$ for all $P \in J_{A,\kappa}$ \cite[Theorem 2.7]{GKZ90}. Nevertheless, it is instructive to prove this using results from \Cref{sec:deRham}.
\begin{lemma} \label{lem:GKZ}
Let $I_A$ and $J_{A,\kappa}$ be as defined above. The $D_A$-ideal $H_{A}(\kappa) \coloneqq I_A + J_{A,\kappa}$ annihilates the function \eqref{eq:cfunction} for any choice of the twisted cycle $\Gamma_{c^*}.$
\end{lemma}
\begin{proof}
We argued above that $I_A$ annihilates $I_{\Gamma_{c^*}}.$ The $i$-th entry of the vector $ A\theta-\kappa,$ with $1 \leq i \leq n,$ is 
$\sum_{j=1}^\ell \sum_{\alpha \in A_j} \alpha_i \, c_\alpha \partial_\alpha + \nu_i.$
Applying \eqref{eq:diffrule}, we compute that 
\begin{align*}
(A \theta - \kappa)_i \bullet I_{\Gamma_{c^*}} \, &\,=\,  \, \sum_{j = 1}^\ell \sum_{\alpha \in A_j} \alpha_i \, c_\alpha \int_{\Gamma_{c^* \rightarrow c}} s_j f^{s-e_j} x^{\alpha + \nu} \frac{\d x}{x} \, + \, \nu_i \int_{\Gamma_{c^* \rightarrow c}}  f^{s} x^{\nu} \frac{\d x}{x} \\
&\,=\, \, \sum_{j=1}^\ell \int_{\Gamma_{c^* \rightarrow c}} s_j f^{s} x^{\nu} \, f_j^{-1} \left( \sum_{\alpha \in A_j} \alpha_i c_\alpha x^\alpha \right)  \frac{\d x}{x} + \, \nu_i \int_{\Gamma_{c^* \rightarrow c}}  f^{s} x^{\nu} \frac{\d x}{x} \\
&\,= \, \left \langle \left [ \left (\sum_{j=1}^\ell s_j f_j^{-1} \frac{\partial f_j}{\partial x_i} x_i   + \nu_i \, \right) \frac{\d x}{x} \right ], [\Gamma_{c^* \rightarrow c}] \right \rangle,
\end{align*}
where we use the pairing between $H_n(X(c),\omega(c))$ and $H^n(X(c),\omega(c))$ seen in \eqref{eq:perfectpairing}. This evaluates to zero by the fact that the cocycle is zero in cohomology: it is $\nabla_{\omega}(\frac{\d x_{\widehat{i}}}{x_{\widehat{i}}}).$ The entry $(A \theta - \kappa)_{n+j}$ is $\sum_{\alpha \in A_j} c_\alpha \partial_\alpha - s_j.$ Using \eqref{eq:diffrule}, one checks that $(A \theta - \kappa)_{n+j}$ annihilates $I_{\Gamma_{c^*}}.$
\end{proof}

\begin{example}($n=2,\ell=1$) \label{ex:gkz}
We consider the polynomial $f\in \CC[x,y]$ defined in \eqref{eq:fcurve}, but replace its coefficients by indeterminates $c_1, \ldots, c_6.$ The matrix $A\in\ZZ^{3\times6}$ in this case is 
\begin{equation} \label{eq:Aexample} A = \begin{pmatrix}
1 & 1 & 2 & 2 & 3 & 3\\
2 & 3 & 1 & 3 & 1 & 2 \\
1 & 1 & 1 & 1 & 1 & 1
\end{pmatrix}. \end{equation}   
Using the {\tt Macaulay2}~\cite{M2} package {\tt Dmodules}~\cite{Dmodm2}, one computes that the toric ideal $I_A$ is generated by $9$ binomials: 
\begin{align*}I_A \,=\, \langle&{\partial}_{2}{\partial}_{5}-{\partial}_{1}{\partial}_{6},{\partial}_{3}{\partial}_{4}-{\partial}_{1}{\partial}_{6},\,{\partial}_{4}{\partial}_{5}^{2}-{\partial}_{3}{\partial}_{6}^{2},\,{\partial}_{1}{\partial
}_{5}^{2}-{\partial}_{3}^{2}{\partial}_{6},{\partial}_{4}^{2}{\partial}_{5}-{\partial}_{2}{\partial}_{6}^{2},\\
& \quad {\partial}_{1}{\partial}_{4}{\partial}_{5}-{\partial}_{2}{\partial}_{3}{\partial}_{6},\,{\partial}_{1}{\partial}_{4}^{2}-{\partial}_{2}^{2}{\partial}_{6},\,{\partial}_{2}{\partial}_{3}^{2}-{\partial}_{1}^{2}{\partial}_{5},{\partial}_{2}^{2}{\partial}_{3}-{\partial}_{1}^{2}{\partial}_{4}\rangle .
\end{align*} 
The ideal $J_{A,\kappa}$ is generated by the $3$ operators $$ \theta_1+\theta_2+2\theta_3+2\theta_4+3\theta_5+3\theta_6+\nu_1, \, 2\theta_1+3\theta_2+\theta_3+3\theta_4+\theta_5+2\theta_6 +\nu_2, \, \theta_1+\theta_2+\theta_3+\theta_4+\theta_5+\theta_6-s.$$
Together, these $12$ operators generate $H_A(\kappa).$
\end{example}

The $D_A$-ideal $H_A(\kappa)$ from \Cref{lem:GKZ} is called a {\em GKZ system} or {\em $A$-hypergeometric system} of degree $\kappa.$ Such systems are examples of regular singular $D$-modules. Solution functions of the GKZ system $H_{A}(\kappa)$ are called {\em $A$-hypergeometric functions}. \Cref{lem:GKZ} implies that our functions \eqref{eq:cfunction} are $A$-hypergeometric. Under some {\em non-resonance} conditions on $\kappa$ (see \Cref{def:nonres}), the converse is also true: all $A$-hypergeometric functions can be written in the form~\eqref{eq:cfunction}. To make this precise, let $\Sol$ be the sheaf of  solutions of the $D_A$-module $D_A/H_A(\kappa)$ and let $\Sol_{c^*}$ be the stalk at $c^*.$ By \Cref{lem:GKZ}, there is a map $H_n(X(c^*),\omega(c^*)) \rightarrow \Sol_{c^*}$ which sends $\Gamma_{c^*}$ to the image of $I_{\Gamma_{c^*}}$ in $\Sol_{c^*}.$ The image of this map is $V_{\Gamma_{c^*}}.$ 
\begin{definition} \label{def:nonres}
A vector $\kappa \in \CC^{n + \ell}$ is {\em non-resonant} if it does not belong to $\CC \cdot F +\ZZ \cdot A $ for any facet $F$ of the cone $\sum_{\alpha \in A} \RR_{\geq 0 } \cdot \alpha$ generated by $A.$ 
\end{definition}
\begin{theorem}\label{thm:chiGKZ}
If $\kappa=(-\nu,s)^\top$ is non-resonant, the $\CC$-linear map $H_n(X(c^*),\omega(c^*)) \to  \Sol_{c^*}$ is an isomorphism, and $\dim_{\CC}(\Sol_{c^*}) = \dim_{\CC}(V_{\Gamma_{c^*}}) = \vert \chi (X(c^*)) \vert.$
\end{theorem}
\begin{proof}
The first claim is \cite[Theorem 2.10]{GKZ90}. The statement about the dimension of $\Sol_{c^*}$ follows from the perfect pairing \eqref{eq:perfectpairing}, \Cref{lem:isoCvectorspaces}, and \Cref{thm:VGamma}.
\end{proof}
By a theorem of Cauchy, Kovalevskaya, and Kashiwara, the dimension of the space of solutions of a $D$-ideal $I$ on any simply connected domain $U$ outside the {\em singular locus} $\Sing(I)$ is equal to the {\em holonomic rank} of $I.$ The definition of the singular locus of a $D$-ideal can be found in \cite[(1.32)]{SST00}. In the case of $I = H_A(\kappa),$ the holonomic rank is given by the dimension of $R_A/(R_A \cdot H_A(\kappa))$ as a $\CC(c)$ vector space, where $\CC(c)$ is the field of rational functions in the coefficients $(c_\alpha)_{\alpha \in A},$ and $R_A$ denotes the Weyl algebra with rational function coefficients. This was outlined in the introduction. The singular locus of our $A$-hypergeometric system is the principal \mbox{$A$-determinant} \cite[Remark 1.8]{gelfand1994discriminants}. We denote this by $\{ E_A (c) = 0 \}.$ 

\begin{remark}
A relevant case in physics is where $f=\mathcal{F}$ is the second \emph{Symanzik polynomial}. The singular locus in this specialization is closely related to the {\em Landau discriminant} from~\cite{mizera2021landau}. Feynman integrals in the Lee--Pomeransky representation were studied using GKZ theory in \cite{Cruz19}. There, $f=\mathcal{U}+\mathcal{F}$ is the sum of the first and second Symanzik polynomial.
\end{remark}

The fact that the dimension of the space of local solutions of $H_A(\kappa)$ is constant on an open dense subset of $\CC^A$ follows from \Cref{thm:chiGKZ} by observing that $\chi(X(c))$ is constant on an open dense subset. We have related this open condition to the principal $A$-determinant in \Cref{sec:deRham}, and described the {\em generic Euler characteristic} as the volume of a polytope ${\NP}(h),$ see \Cref{thm:VGamma}. It is not surprising that the same polytope makes an appearance here.

\begin{theorem}\label{thm:dim_GKZ}
Let $c^* \in \CC^A$ be such that $E_A(c^*) \neq 0$ and let $\kappa$ be non-resonant. For any simply connected domain $U_{c^*} \ni c^*$ such that $U_{c^*} \cap \{E_A(c) = 0 \} = \emptyset,$ we have that 
\[ \dim_{\CC}\left(V_{c^*} \right) \, = \, \dim_{\CC(c)} \left( R_A/(R_A \cdot H_A(\kappa))\right) \, = \, |\chi(X(c^*))| \, =  \, {\vol}\left({\NP}(h)\right), \]
where $V_{c^*}$ is defined as in \eqref{eq:Vc}, and $\NP (h)$ is the Newton polytope of $h$ from \eqref{eq:h}.
\end{theorem}
\begin{remark}
The quantity $\vol \left({\NP}(h)\right)$ in \Cref{thm:dim_GKZ} is the $(n+\ell-1)$-dimensional \emph{normalized volume} of our polytope with respect to the rank-$(n+\ell-1)$ sublattice of $\mathbb{Z}^{n + \ell}$ affinely generated by the columns of $A$, see \cite[\S 5.3]{GKZ90}. 
If this is the same as the lattice obtained by intersecting the affine span of ${\rm NP}(h)$ with $\mathbb{Z}^{n + \ell}$, the volume equals $(n+\ell)!$ times the Euclidean volume of ${\rm Conv}(\{0\} \cup {\rm NP}(h))$. 
\end{remark}
\begin{example}
Using the command {\tt holonomicRank} in {\tt Macaulay2}, we check that the holonomic ideal $H_A(\kappa)$ from \Cref{ex:gkz} is $6.$ This number coincides with the degree of the toric ideal $I_A$ associated to the matrix $A$ from \eqref{eq:Aexample}, and the normalized volume of the hexagonal pyramid from \Cref{fig:polytopes}.
\end{example}

We point out that, for any choice of parameters $\kappa = ( -\nu, s)^\top,$ we have the inequality
\begin{align} 
\dim_{\CC(c)} \left( R_A/(R_A \cdot H_A(\kappa))\right) \, \geq \, {\vol}({\NP}(h)),
\end{align} 
see for instance \cite[Theorem 3.5.1]{SST00}. 
Equality holds for non-resonant parameters, but the holonomic rank may {\em jump up} for special $\kappa.$ The Euler characteristic of $X(c^*)$ may {\em drop} for special choices of the coefficients $c^*$; this is what happened in Example \ref{ex:lines}.

Following \cite[Sections 3, 4]{Matsubara2020}, we now explain which $D_A$-modules are behind these constructions. This will lead to an explicit connection between this section and \Cref{sec:deRham}.~Define
\begin{align} 
{\cal X}  \, \coloneqq \, \left\{ (x,c) \in (\CC^{\ast})^n \times \CC^A \,\mid \, \prod_{i=1}^\ell f_i(x; c) \neq 0 \right\}
\end{align} 
and let $\pi \colon {\cal X} \to \CC^A$ be the projection to $\CC^A$. For $c\in \CC^A,$ the fiber of $\pi$ is $X(c).$
Note that $H^n(X(c),\omega(c))$ depends rationally on the $c_i.$ 
Denote by $\cH^n$ the $n$-th cohomology group of the relative de Rham complex $(\Omega_{{\cal X}/\CC^A}^{\bullet},\nabla_x),$ where $\Omega_{{\cal X}/\CC^A}^k$ is the sheaf of relative differential $k$-forms. This sheaf is locally defined by its sections
$ \sum_{\vert I \vert = k} s(x,c)\,\d x^I,$ where $s(x,c)$ are sections of the structure sheaf $\mathcal{O}_{\cal X}.$ The differential $\nabla_x = \d_x +\dlog_x (f^s x^{\nu})$ only takes derivatives with respect to $x.$ For every $c^* \in \CC^A,$ there is an evaluation map 
\begin{align} \label{eq:ev}
\ev_{c^*}\colon \cH_{c^*}^n \to H^n(X(c^*),\omega(c^*)) .
\end{align}
We now recall that $\cH^n$ is naturally endowed with the structure of a $D_A$-module via the {\em Gau{\ss}--Manin connection} $\nabla^{\GM}\coloneqq \nabla_c=\d_c+\dlog_c (f^sx^{\nu}).$

Denote by $M_A(\kappa)$ the regular holonomic $D_A$-module $D_A/H_A(\kappa).$

\begin{proposition}[\cite{Matsubara2020}]\label{prop:GKZGM}
For non-resonant $\kappa$ (see \Cref{def:nonres}) and $s\notin \ZZ^\ell,$ the morphism
$$ M_A(\kappa) \stackrel{\cong}{\longrightarrow} \cH^n,
\quad \left[ 1 \right] \mapsto \left[ \frac{\d x}{x}\right]$$ 
is an isomorphism of $D_A$-modules. 
\end{proposition}
The isomorphism from Proposition \ref{prop:GKZGM} explains how to use GKZ theory to obtain relations between the generators of $V_\Gamma$. It extends $[1] \mapsto \left[ \frac{\d x}{x}\right]$ $D_A$-linearly. Explicitly, $[P]\in M_A(\kappa)$ is sent to $P\bullet[\frac{\d x}{x}],$ where $\partial_\alpha$ acts on an element $[\phi(c)]\in \cH^n$ as follows (cf. \cite{MT21}):
\begin{align} 
\partial_\alpha \bullet \left[\phi(c)\right]\,\,=\,\, \left [ \partial_{\alpha}\left(\phi(c)\right)+ \left ( \sum_{j=1}^{\ell}  s_j \frac{x^{\alpha}}{f_j(x;c)} \right ) \phi(c) \right ].
\end{align}
The image of an element in $H_A(\kappa)$ is zero in $\cH^n$ by \Cref{prop:GKZGM}. Applying the evaluation map \eqref{eq:ev} gives a zero-element in $H^n(X(c^*),\omega(c^*)).$
\begin{example}
The image of the differential operator $(A\theta - \kappa)_i$ under the isomorphism in \Cref{prop:GKZGM} is 
\[ \left [ \left (\sum_{j=1}^\ell s_j f_j^{-1} \frac{\partial f_j}{\partial x_i} x_i   + \nu_i \, \right) \frac{\d x}{x} \right ], \]
which is precisely the zero-cocycle seen at the end of the proof of \Cref{lem:GKZ}.
\end{example}

\section{Numerical nonlinear algebra} \label{sec:numerical}
In previous sections, we have focused on symbolic techniques for computing with generalized Euler integrals. We now switch gears and present some ideas for the use of numerical methods. We believe that integrating these different approaches will be of key importance to compute larger instances. A first, well-known observation is that the dimension of our vector spaces can be computed by solving a system of rational function equations. Fix $f \in \CC[x, x^{-1}]^\ell,$ $s \in \CC^\ell,$ and $\nu \in \CC^n,$ and let $X$ be as in \eqref{eq:X_intro}. We say that $x^* \in X$ is a {\em complex critical point} of $\log(f^s \, x^{\nu})$ on $X$ if $\omega(x^*) = 0,$ where $\omega$ is the $1$-form defined in~\eqref{eq:omega}.

\begin{theorem} \label{thm:critpoints}
Fix $f \in \CC[x,x^{-1}]^\ell$ and let $X$ be as in \eqref{eq:X_intro}. The integer $(-1)^n \cdot \chi(X)$ equals the number of complex critical points of $\log(f^s \, x^{\nu})$ on $X,$ for generic $s \in \CC^\ell,$ $ \nu \in \CC^n.$ 
\end{theorem}
\begin{proof}
Since $X$ is smooth and very affine, this follows from \cite[Theorem 1]{Huh}.
\end{proof}
Concretely, one obtains the Euler characteristic of $X$ by counting complex solutions of 
\begin{equation} \label{eq:criteq}
    \frac{s_1 \cdot \frac{\partial f_1}{\partial x_i}}{f_1} \,+\, \cdots \,+\, \frac{ s_\ell \cdot  \frac{\partial f_\ell}{\partial x_i}}{f_\ell} \,+\, \frac{\nu_i}{x_i} \, = \, 0 , \qquad i = 1, \ldots, n. 
\end{equation}  
This has been applied to count master integrals in \cite{mizera2021landau}, and to compute Euler characteristics of point configuration spaces in \cite{agostini2021likelihood, sturmfels2021likelihood} with a view towards physics and statistics.
One way to solve the equations in \eqref{eq:criteq} is by using numerical homotopy methods. For the computations in this paper, we use the {\tt Julia} package {\tt HomotopyContinuation.jl} (v2.6.3) \cite{breiding2018homotopycontinuation}.
\begin{example}
We compute the Euler characteristic of the very affine surface $X$ from \Cref{ex:lines}, with $f$ as in \eqref{eq:fcurve}. The equations in \eqref{eq:criteq} are generated in {\tt Julia} as follows: 
\begin{verbatim}
using HomotopyContinuation
@var x y s ν[1:2]
f = -x*y^2 + 2*x*y^3 + 3*x^2*y - x^2*y^3 - 2*x^3*y + 3*x^3*y^2
L = s*log(f) + ν[1]*log(x) + ν[2]*log(y)
F = System(differentiate(L,[x;y]), parameters = [s;ν])
\end{verbatim}
The variable {\tt F} is viewed as a system of two equations in two unknowns {\tt x} and {\tt y}, parameterized by {\tt s}, {\tt ν[1]}, and {\tt ν[2]}. Solving for generic parameter values is done using the command {\tt monodromy\_solve(F)}. The output confirms that there are $6$ complex solutions. We encourage the reader to check that the analogous computation with $f$ as in~\eqref{eq:flines} returns only $2$ solutions. The equations~\eqref{eq:criteq} in this case coincide with \cite[Equation (5)]{sturmfels2021likelihood} for $m = 5.$ For a tutorial on solving \eqref{eq:criteq} in {\tt Julia}, see \cite[Section 3]{sturmfels2021likelihood} and the references therein.
\end{example}
Another use for these tools is the computation of $\CC$-linear relations among the generators 
\begin{equation} \label{eq:Iab}
    \left (I_{a,b}: [\Gamma] \longmapsto \int_\Gamma f^{s+a} \, x^{\nu + b} \, \frac{\d x}{x} \right ) \, \in \, \Hom_{\CC}(\, H_n(X,\omega), \CC \,)
\end{equation}
of our vector space $V_\Gamma,$ discussed in \Cref{sec:deRham}. We will present the ideas in the case where $n=1$ and leave general methodology for future research. In particular, we will reproduce the relation found in \Cref{ex:elliptic}. We start with a discussion on numerically computing the integral $I_{a,b}(\Gamma).$ Since $\Gamma$ is a {\em twisted} $1$-cycle, it is  encoded by a singular $1$-cycle $\Delta$ and a choice of branch $\phi$ for the multi-valued function $f^s x^{\nu}.$ We will take $\Delta$ to be a triangle $ABC,$ i.e.,~a sum $AB + BC + CA$ of line segments, with $A,B,C \in X \subset \CC^*.$ The cycle $\Gamma$ is 
\begin{equation} \label{eq:Gammatriangle}
\Gamma \, = \, AB \otimes_{\CC}\phi_{AB} \,+\, BC \otimes_{\CC}\phi_{BC} \,+\, CA \otimes_{\CC}\phi_{CA},
\end{equation}
where $\phi_{AB}: U_{AB} \rightarrow \CC$ is a section of ${\cal L}_\omega^\vee,$ defined on the open neighborhood $U_{AB}$ of the line segment ${AB},$ and similarly for $\phi_{BC}$ and $\phi_{CA}.$ Note that $\phi_{AB}$ is completely determined by its value $\phi_{AB}(A)$ at $A,$ and $\phi_{AB}(B) = \phi_{BC}(B), \phi_{BC}(C) = \phi_{CA}(C).$ Moreover, we also have $\phi_{CA}(A) = \phi_{AB}(A)$ because $\Gamma$ is a twisted cycle. 
Hence, the data specifying $\Gamma$ are the points $A, B, C,$ and the complex number $\phi_{AB}(A).$ The integral is 
\[ \int_\Gamma f^{s+a} \, x^{\nu + b} \, \frac{\d x}{x} \,\, = \,\,  \int_{AB} \phi_{AB}(x) f^a x^b \frac{\d x}{x} \,+\,\int_{BC} \phi_{BC}(x) f^a x^b \frac{\d x}{x} \,+\, \int_{CA} \phi_{CA}(x) f^a x^b \frac{\d x}{x},\]
where the three integrals on the right are usual complex integrals over singular 1-chains on~$X,$ with a single-valued integrand. These can be approximated using the trapezoidal rule. Fixing a large integer $N$ and writing $x_i = (N-1)^{-1} \cdot ((N-i)A + (i-1)B),$ we get
\begin{equation} \label{eq:trap}
\int_{AB} \phi_{AB}(x) f^a x^b \frac{\d x}{x} \,\, \approx \,\, \frac{(B-A)}{N-1}  \left ( \frac{\phi_{AB}(x_1)\psi_1}{2} +  \sum_{i=2}^{N-1} \phi_{AB}(x_i)\psi_i + \frac{\phi_{AB}(x_N)\psi_N}{2} \right ),
\end{equation}
where $\psi_i \coloneqq {f(x_i)}^{a} x_i^{b-1}$ is the evaluation of the single-valued part of the integrand at $x_i.$ 

To evaluate \eqref{eq:trap}, we need to evaluate the section $\phi_{AB}$ at the nodes $x_i \in AB$ of the numerical integration. To this end, recall that $\phi_{AB}$ satisfies a differential equation 
\begin{equation} \label{eq:diffeq}
\frac{\d \phi_{AB}(x) - \omega \cdot \phi_{AB}(x) }{\d x} \,=\, 0, \quad x \in AB, 
\end{equation}
with initial condition specified by $\phi_{AB}(x_1) = \phi_{AB}(A).$ One can use any standard method for numerically solving ODEs to approximate $\phi_{AB}(x_i),$ $i =1 , \ldots, N.$ Furthermore, $\phi_{AB}(x_N) = \phi_{AB}(B)$ can be used as the initial condition for the next integral over the line segment $BC.$

When the parameters $s,$ $\nu$ are rational numbers, we can make use of the fact that the graph $(x,\phi_{AB}(x)),$ $x \in AB,$ satisfies an algebraic equation $F(x,y) = 0.$ Indeed, let $k$ be the smallest integer such that $k \nu \in \ZZ^n$ and $k s \in \ZZ^\ell.$ We have $F(x,\phi_{AB}(x)) = \phi_{AB}(x)^k - f(x)^{ks} x^{k\nu} = 0.$ Consider the algebraic curve ${\cal C} = \{ (x,y) \in (\CC^*)^2 \,|\, y^k - f(x)^{ks}x^{k\nu} = 0\}$ with marked points $Z = \{(x,0) \,|\, f(x) = 0 \} .$ There is a degree $k$ covering 
\[ \pi \, : \, {\cal C} \setminus Z \longrightarrow X,\quad (x,y) \mapsto x. \]
Suppose that, using \eqref{eq:diffeq}, we have computed an approximation $\tilde{y}_i$ for $y_i \coloneqq \phi_{AB}(x_i) \in \pi^{-1}(x_i).$ We can use $\tilde{y}_i$ as a starting point for Newton iteration on the nonlinear equation ${F(x_i,y) = 0}$ in the variable $y.$ If $\tilde{y}_i$ is a reasonable approximation, the iteration will converge to $y_i,$ and reduce the approximation error of our numerical ODE solver significantly in each discretization step. This is illustrated in \Cref{fig:predict_correct}. The procedure is much like the standard {\em predict-and-correct} technique used in polynomial homotopy continuation, see \cite[Chapter 3]{allgower2012numerical}.\\
\begin{figure}
\centering
\includegraphics[width = 10cm]{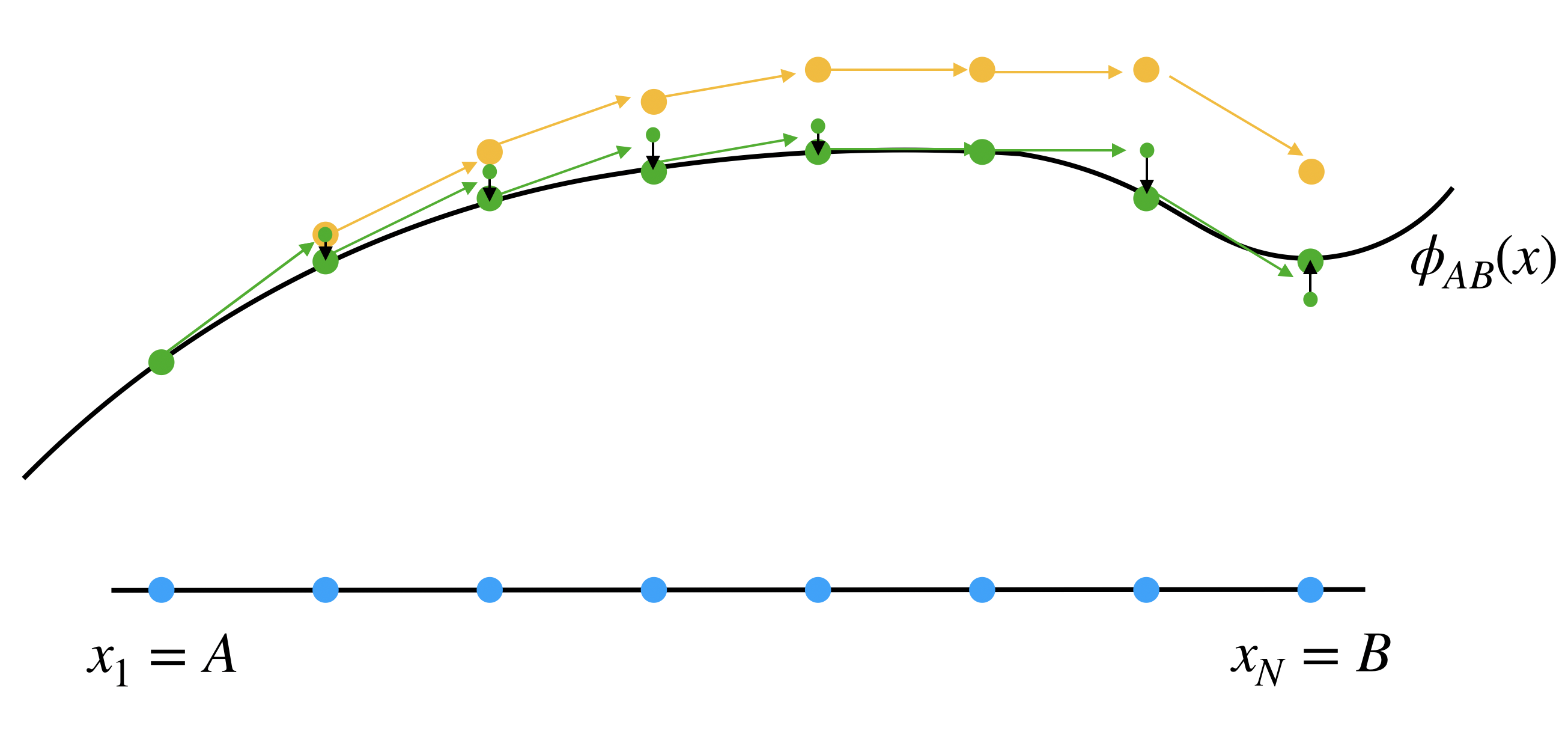}
\caption{Estimating $\phi_{AB}(x_i)$ using a numerical ODE solver (yellow) with initial condition at $x_1 = A.$ Results are improved by adding Newton iterations in each step (green).}
\label{fig:predict_correct}
\end{figure}
Suppose we know a basis of twisted cycles $[\Gamma_1], \ldots, [\Gamma_\chi]$ for $H_1(X,\omega),$ where $\Gamma_i$ are as in~\eqref{eq:Gammatriangle} and $\chi = |\chi(X)|.$ Given $\chi+1$ cocycles $[f^{a^{(1)}} x^{b^{(1)}} \frac{\d x}{x}], \ldots,  [f^{a^{(\chi+1)}} x^{b^{(\chi+1)}} \frac{\d x}{x}] \in H^1(X,\omega),$ we would like to compute a $\CC$-linear relation between the corresponding $I_{a^{(j)}, b^{(j)}}$ from \eqref{eq:Iab}.
Let $M_{ij} \coloneqq I_{a^{(j)},b^{(j)}}(\Gamma_i) =  \langle [f^{a^{(j)}} x^{b^{(j)}} \frac{\d x}{x}], [\Gamma_i] \rangle \in \CC$ be given by the perfect pairing \eqref{eq:perfectpairing}. These are approximated numerically using the techniques outlined above. We then arrange these numbers in a matrix $M = (M_{ij})_{1 \leq i \leq \chi, 1 \leq j \leq \chi+1}.$ 

\begin{proposition}
Any vector $(c_1, \ldots, c_{\chi+1})$ in the kernel of $M,$ viewed as a linear map $\CC^{\chi+1} \rightarrow \CC^\chi,$ gives a linear relation $\sum_{j=1}^{\chi+1} c_j \cdot I_{a^{(j)}, b^{(j)}} = 0.$ 
\end{proposition}
\begin{proof}
For any twisted cycle $[\Gamma] = d_1 [\Gamma_1] + \cdots + d_\chi [\Gamma_\chi]$ we have 
\[
    \sum_{j=1}^{\chi+1} c_j \,  I_{a^{(j)},b^{(j)}}(\Gamma) \, = \, \sum_{i=1}^\chi \sum_{j=1}^{\chi+1} d_i c_j I_{a^{(j)},b^{(j)}}(\Gamma_i) \, = \, \sum_{i=1}^\chi d_i \sum_{j=1}^{\chi+1}  c_j M_{ij} \,=\, 0.  \qedhere \]
\end{proof}

\begin{example}[Examples \ref{ex:elliptic0} and \ref{ex:elliptic} continued]
Let $f = (x-1, x-2),$ $s = (\sfrac{1}{2}, \sfrac{1}{2})$ and $\nu = \sfrac{1}{2}$ be as in Examples \ref{ex:elliptic0} and \ref{ex:elliptic}. The green cycles in \Cref{fig:torus} form a basis for $H_1(X,\omega).$ We replace them by triangles $\Gamma_1, \Gamma_2$ as in \eqref{eq:Gammatriangle}. These are specified by the data 
\small
\[ \begin{matrix}
A_1 = \sfrac{1}{2}+\sqrt{-1}, & B_1 = \sfrac{1}{2} - \sqrt{-1}, & C_1 = 3, & \phi_{A_1B_1}(A_1) = -1.436744 + 0.435011 \sqrt{-1}, \\
A_2 = -1,& B_2 = \sfrac{3}{2} + \sqrt{-1}, & C_2 = \sfrac{3}{2} - \sqrt{-1}, & \phi_{A_2B_2}(A_2) = -2.449490 \sqrt{-1}.
\end{matrix}
\]
\normalsize
We integrate against the cocycles $[f^{a^{(j)}} x^{b^{(j)}} \frac{\d x}{x}],$ $j = 1, \ldots, 3 ,$ with $a^{(1)} = (-1,0),$ $a^{(2)} = (0,-1),$ $a^{(3)} = (0,0)$ and $b^{(1)} = 1,$ $b^{(2)} = 1,$ $b^{(3)} = 0.$ We do this using an implementation in Julia of the ideas discussed above. The code can be found in \Cref{app:code}. Here is how to use it in this particular example: 
\begin{verbatim}
f = x -> [x-1; x-2]; s = [1/2;1/2]; ν = 1/2; N = 1000; k = 2;
ω = x -> s[1]/(x-1) + s[2]/(x-2) + ν/x
cocycles = [[[-1;0],1], [[0,-1],1], [[0,0],0]]
\end{verbatim}
\vspace*{-1mm}
\begin{verbatim}
A1 = 1/2+im; B1 = 1/2-im; C1 = 3+0im
phiA1B1_at_A1 = A1^ν*prod(f(A1).^s)
I1 = integrate_loop(A1,B1,C1,phiA1B1_at_A1,N,f,ω,s,ν,k,cocycles)
\end{verbatim}
\vspace*{-1mm}
\begin{verbatim}
A2 = -1+0im; B2 = 3/2+im; C2 = 3/2-im
phiA2B2_at_A2 = A2^ν*prod(f(A2).^s)
I2 = integrate_loop(A2,B2,C2,phiA2B2_at_A2,N,f,ω,s,ν,k,cocycles)
\end{verbatim}
The variable ${\tt I1}$ contains the first row $(M_{11} \  M_{12} \ M_{13})$ of our $2 \times 3$ matrix $M.$ The second row is ${\tt I2}.$ We obtain the matrix
\[ M \, = \, \begin{pmatrix}
-3.496 \sqrt{-1} & 4.144 \sqrt{-1} & - 0.648 \sqrt{-1}\\
3.496 & 0.648 & -4.144
\end{pmatrix}, \]
whose kernel is spanned by $(\sfrac{1}{2},\sfrac{1}{2},\sfrac{1}{2})^\top.$ This is the relation seen in \Cref{ex:elliptic}.
\end{example}

\subsection*{Acknowledgments} 
Simon Telen was supported by a Veni grant from the Netherlands Organisation for Scientific Research (NWO). We are grateful to Johannes Henn, Saiei-Jaeyeong Matsubara-Heo, Sebastian Mizera, Bernd Sturmfels, and Nobuki Takayama for insightful discussions. 
We thank two anonymous referees for their insightful comments on an earlier version of this article. 

\appendix 
\section{Vanishing of cohomology groups (Appendix by Saiei-Jaeyong Matsubara-Heo)}\label{appendixA}

In this appendix, we prove a vanishing theorem for the twisted de Rham cohomology of very affine varieties. Afterwards, we see how to adapt it to the situation of the paper. Let $(\CC ^*)^N$ be an algebraic $N$-torus with coordinates $z_1,\dots,z_N,$ and let $U\subseteq (\CC ^*)^N$ be a smooth closed subvariety of dimension $n.$
For any complex vector $\alpha = (\alpha_1,\dots,\alpha_N)$ we define the following $1$-form on $U$:
\begin{align}
 \omega \,\coloneqq \, \dlog \left(z_1^{\alpha_1}\cdots z_N^{\alpha_N}\right) \,=\, \sum_{j=1}^N \alpha_j \cdot  \dlog \left( z_j \right),
\end{align}
where $d$ is the exterior derivative on $U.$ As in \Cref{sec:deRham}, the algebraic de Rham cohomology groups on variety $U$ are defined as
\begin{align}\label{eq:algDR}
H^k(U,\omega) \, \coloneqq \, H^k(\Omega_U^{\bullet}(U), \nabla_{\omega}), 
\end{align} 
where, $\Omega^{\bullet}_U(U)$ is the complex of algebraic differential forms on $U$ with the twisted differential $\nabla_{\omega} = \d  + \omega \wedge .$ Note that it is enough to take the cohomology of the complex $\Omega^{\bullet}_U(U)$  of global sections since the variety $U$ is affine (see \cite[Corollaire 6.3]{Del}).

The purpose of this appendix is to prove the vanishing result of Theorem \ref{thm:vanishingcohomology} below. We thank the anonymous referee for suggesting that this could be derived from \cite[Theorem~3.4.4]{dimca2004sheaves}, together with Huh's compactification of $U$ presented in \cite[\S 2.3]{Huh}. We include a proof nonetheless, in order to have a self-contained argument, and in particular because it sheds light on the meaning of genericity for $\alpha$.

\begin{theorem}\label{thm:vanishingcohomology}
For generic $\alpha,$ the twisted algebraic de Rham cohomology is purely \mbox{$n$-codimensional},~i.e.,
\begin{equation}
H^i(U,\omega)\,=\, 0 \quad \text{for all }i\neq n.
\end{equation}
\end{theorem}
\begin{proof}
The idea of the proof is to show that for a generic $\alpha$ we have that 
\begin{equation}\label{eq:twistedduality}
H^i\left(U,\omega\right) \,\cong \, H^{2n-i}(U,-\omega)^{\vee} 
\end{equation}
so that the vanishing follows from \eqref{eq:algDR} and the fact that the algebraic de Rham complex $\Omega^{\bullet}_U(U)$ has no terms of degree higher than the dimension. To do this, we will work on the analytic variety~$U^{\an }.$
Thus, let $\cL _{\omega}$ be the local system of analytic flat sections of the twisted differential  \mbox{$\nabla_{\omega} \colon \mathcal{O}_{U^{\an}} \to \Omega^1_{U^{\an}}.$} To compare this with the notation of Section \ref{sec:deRham}, it holds that $\mathcal{L}^{\vee}_{\omega} \cong \mathcal{L}_{-\omega}$. By the Deligne--Grothendieck comparison theorem (cf. \cite[Corollaire 6.3]{Del}), we have a canonical isomorphism 
  $H^i(U,\omega) \, \cong\, H^i(U^{\an},\cL _\omega)$.
In \Cref{sec:deRham}, the vector space on the right is denoted $H_{\dR}(U,\omega).$ Poincar\'{e} duality yields a canonical isomorphism to the dual of the twisted cohomology with compact support:
\begin{equation}\label{eq:poincare} 
H^i(U^{\an },\cL _{\omega}) \,\cong\, H^{2n-i}_c(U^{\an },\cL _{-\omega})^{\vee}.  
\end{equation}
\Cref{lemma:compactsupport} proves that
    $H^{2n-i}_c(U^{\an},\cL _{-\omega}) \,\cong\, H^{2n-i}(U^{\an},\cL _{-\omega})$.
Again, by the Deligne--Grothendieck comparison theorem, $H^{2n-i}(U^{\an},\cL _{-\omega})$ is isomorphic to $H^{2n-i}(U,-\omega)$.
\end{proof}

\begin{lemma}\label{lemma:compactsupport}
Let $\alpha \in \CC^N$ be general. With the above notation, we have
\[ H^i_c(U^{\an },\cL _{\omega}) \,\cong \, H^i(U^{\an },\cL _{\omega}) \quad \text{ for all } i. \]
\end{lemma}

\begin{proof}
The key ingredient is a compactification of $U$ constructed by Huh~\cite[\S2.3]{Huh}. This is an open embedding $\iota\colon U \hookrightarrow \overline{U}$ in a smooth projective variety $\overline{U}$ such that the boundary $D=\overline{U}\setminus U$ is a simple normal crossing divisor and, most importantly, if $D'\subseteq D$ is a component of $D,$ then there is at least one coordinate function $z_j$ such that 
\begin{equation}\label{eq:ordernonvanishing}
\ord_{D'}(z_j) \,\ne\, 0.
\end{equation}
Here, we consider the coordinate function $z_j\colon U \to \CC ^*$ as a rational function on~$\overline{U}.$

It will be convenient to use the language of derived functors. Furthermore, in the rest of the proof, we will work in the analytic category, without denoting it explicitly. Thus, consider the direct image functor $\iota_{*}$ and the direct image with compact support functor $\iota_{!}.$ Looking at the composition $U \stackrel{\iota} {\hookrightarrow}\overline{U} \overset{s}{\to} \{ \operatorname{pt} \},$ we~see~that $\Gamma_{U,c} = (s\circ \iota)_{!} = s_{!} \circ \iota_{!} =  \Gamma_{\overline{U}} \circ \iota_{!}$ and $\Gamma_{U} = (s\circ \iota)_* = s_* \circ \iota_* =  \Gamma_{\overline{U}} \circ \iota_{*},$
where $\Gamma_U$ and $\Gamma_{U,c}$ denote the functors of taking sections and sections with compact support, respectively. By taking the derived functors, we see that 
\[ H^i_c(U,\cL _{\omega}) \,\cong \,H^i( (R\Gamma_{\overline{U}} \circ R\iota_{!})(\cL _{\omega})) \quad \text{and} \quad H^i(U,\cL _{\omega}) \,\cong \, H^i( (R\Gamma_{\overline{U}} \circ R\iota_{*})(\cL _{\omega})).  \]
Thus, it is enough to prove that the canonical morphism \begin{equation}\label{eq:mapderived}
    R\iota_{!}(\cL _{\omega}) \to R\iota_{*}(\cL _{\omega})
\end{equation}
is an isomorphism. This is where the properties of the compactification $\overline{U}$ come into play, and we can use a well-known strategy, see for example \cite[Lemma 3]{CDO}. To check that \eqref{eq:mapderived} is an isomorphism, we check it at the level of the stalks. This is true at the stalk of a point $x\in U,$ since $U$ is open, and if $x\in D$ we see that the stalk on the left hand side is zero by proper base change. Thus, we need to prove that the stalk on the right is zero as well, meaning that 
\[ H^i(V, R\iota_*\cL _{\omega} ) \,\cong \, H^i(V\setminus D, \cL _{\omega}) \,=\, 0 \quad \text{for all } i, \]
where $V$ is a small open neighborhood of $x$ in $\overline{U}.$

Since $D$ is a normal crossing divisor, we can find a small open neighborhood $V$ of $x$ in $\overline{U}$ with analytic coordinates $x_1,\dots,x_N$ centered at $x$ such that $D\cap V = \{ x_1\cdots x_k = 0 \}.$ We can also assume that $V$ is a product of disks. Hence, using the homotopy invariance of cohomology with local coefficients, together with K\"unneth's formula, we see that
\begin{align*}
H^i(V\setminus D , \cL _{\omega}) \,\cong\,  \bigoplus_{i_1+\dots+i_k \,=\, i} H^{i_1}(S^1,\cL _{1}) \otimes \dots \otimes H^{i_k}(S^1,\cL _{k}).
\end{align*}
The local systems $\cL _j$ are determined by the residue of the connection $\nabla_{\omega}$ along the divisor $D_h = \{ x_h=0 \},$ which is given by $\Res _{D_h}(\nabla_{\omega}) = \sum_{j=1}^N \alpha_j \cdot \ord_{D_h}(z_j).$ By construction \eqref{eq:ordernonvanishing}, this is a nonzero integer combination of the $\alpha_j$'s, and since these are generic, it follows that $\Res _{D_h}(\nabla_{\omega})$ it not an integer. Hence, the local systems $\cL _h$ on $S^1$ are nontrivial and $H^i(S^1,\cL _h)=0$ for all~$i.$
\end{proof}

\begin{remark}\label{rmk:alpha_generic}
The proof of \Cref{thm:vanishingcohomology} shows that the conclusion is true whenever the entries of $\alpha$ are linearly independent over~$\QQ.$ However, this is not a necessary condition. The precise meaning of ``generic'' in \Cref{thm:vanishing} is that $\Res_{D_h}$ is required not to be an integer for any boundary divisor~$D_h$.
\end{remark}

Finally, we show how \Cref{thm:vanishingcohomology} translates to \Cref{thm:vanishing} in the setting of our article.

\begin{proof}[Proof of \Cref{thm:vanishing}] Let $x=(x_1,\dots,x_n)$ be coordinates on an algebraic torus $(\CC ^*)^n$ and let $f_1,\dots,f_{\ell} \in \CC [x_1^{\pm 1},\ldots,x_n^{\pm 1}]$ be Laurent polynomials. We consider the very affine variety $X = (\CC ^*)^n \setminus V(f_1\cdots f_{\ell})$
and the graph embedding
\[ X \hookrightarrow (\CC ^*)^{\ell} \times (\CC ^*)^n, \quad x \mapsto \left( f_1(x),\dots,f_{\ell}(x),x_1,\dots,x_n\right). \]
This map is a closed embedding that identifies $X$ with the smooth closed subvariety 
\[ U \,=\, \left\{ (y,x) \in (\CC ^*)^{\ell+n} 
\mid \, y_j = f_j(x) \,\, \text{ for } j=1,\dots,\ell \right\}  \,\subset\, (\CC ^*)^{\ell+n} . \]
As in the setting of \Cref{thm:vanishingcohomology}, consider $\alpha = (s_1,\dots,s_{\ell},\nu_1,\dots,\nu_n).$ Then the differential form $\eta = \sum_{i=1}^{\ell} s_j \cdot \dlog (y_j) + \sum_{j=1}^n \nu_j \cdot \dlog (x_j)$ on $U$ corresponds to the differential form $\omega = \dlog(f^s x^{\nu})$ on $X.$ Hence $H^i(X,\eta) \cong H^i(U,\omega)$ and the claim follows by~\Cref{thm:vanishingcohomology}. 
\end{proof}

The fact that the Euler integral in consideration
\begin{align}
\int_\Gamma f_1(x)^{-s_1}\cdots f_\ell(x)^{-s_\ell}x^\nu \, \d x_1\wedge\cdots\wedge \d x_n 
\end{align}
involves the toric factor $x^\nu$ is key in order to apply \Cref{thm:vanishingcohomology}, since it allows us to translate the problem to the setting of very affine varieties.
In general, the twisted cohomology groups associated to an Euler integral of the~form
\begin{align}
\int_\Gamma f_1(x)^{-s_1}\cdots f_\ell(x)^{-s_\ell}\, \d x_1\wedge\cdots\wedge \d x_n
\end{align}
may not satisfy the vanishing theorem.
For example,  if $n=2$, $\ell=1$, and $f_1=1-x_1x_2$, one can check that $x_2\d x_1+x_1\d x_2$ is a nonzero class in the first twisted cohomology group for any~$s_1$.

\begin{remark} \label{rem:discussion}
The vanishing theorem of twisted cohomology groups for Euler integrals has appeared in several specialized contexts. Most versions consider the case of hyperplane arrangements~\cite{CDO,ESV}. Some partial extensions to Euler integrals with polynomials of higher degree appear in~\cite{KN}.
\end{remark}

\section{Julia code} \label{app:code}
This appendix contains the {\tt Julia} code used for the experiment in~\Cref{sec:numerical}. The notation follows that of the paper closely. The code can be run with {\tt Julia}, version {\tt 1.7.1}. It consists of six functions, copied below with a few lines of documentation. 

\footnotesize
\begin{verbatim}
# One step of Euler's method on dy/dx = ω(x)y, with stepsize Δx.
function euler_step(x,y,Δx,ω)
    newy = (1+ω(x)*Δx)*y
    return (x+Δx,newy)
end

# One step of Newton iteration on y^k - prod(f(x).^(k*s))*x^(k*ν).
function newton_step(y, x, f, s, ν, k)
    Fy = y^k - prod(f(x).^(k*s))*x^(k*ν)
    Δy = -Fy/(k*y^(k-1))
    return y + Δy
end

# Compute function values at N equidistant nodes on the line segment SxTx.
# The initial condition is given by Sy.
function track_line_segment(Sx,Sy,Tx,N,f,ω,s,ν,k)
    xx = zeros(ComplexF64,N)
    yy = zeros(ComplexF64,N)
    xx[1] = Sx
    yy[1] = Sy
    Δx = (Tx - Sx)/(N-1)
    for i = 2:N
        (xx[i], yi_tilde) = euler_step(xx[i-1],yy[i-1],Δx,ω)
        for j = 1:4
            yi_tilde = newton_step(yi_tilde, xx[i], f, s, ν, k)
        end
        yy[i] = copy(yi_tilde)
    end
    return xx, yy
end

# Numerical integration based on function values yy and interval length h.
function integrate_trapezoidal(yy,h)
    return (yy[1]/2 + sum(yy[2:end-1]) + yy[end]/2)*h
end

# Compute the integral over the line segment AB with N discretization nodes. 
# The initial condition is given by phiAB_at_A.
# The integral is computed for the cocycles in aabb. 
function integrate_line_segment(A,phiAB_at_A,B,N,f,ω,s,ν,k,aabb)
    IAB = []
    xxAB, yyAB = track_line_segment(A,phiAB_at_A,B,N,f,ω,s,ν,k)
    for ab in aabb
        a = ab[1]; b = ab[2];
        single_valued = [x^(b-1)*prod(f(x).^a) for x in xxAB]
        IAB = push!(IAB,integrate_trapezoidal(yyAB.*single_valued,(B-A)/(N-1)))
    end
    return IAB,yyAB
end

# Compute the integral over the twisted cycle defined by A, B, C,
# and the initial condition phiAB_at_A. 
# The integral is computed for the cocycles in aabb. 
function integrate_loop(A,B,C,phiAB_at_A,N,f,ω,s,ν,k,aabb)
    IAB, yyAB = integrate_line_segment(A,phiAB_at_A,B,N,f,ω,s,ν,k,aabb)
    phiBC_at_B = yyAB[end]
    IBC, yyBC = integrate_line_segment(B,phiBC_at_B,C,N,f,ω,s,ν,k,aabb)
    phiCA_at_C = yyBC[end]
    ICA, yyCA = integrate_line_segment(C,phiCA_at_C,A,N,f,ω,s,ν,k,aabb)
    return IAB+IBC+ICA    
end
\end{verbatim}


\vfill 

\noindent{\bf Authors' addresses:}
\smallskip

\noindent Daniele Agostini, Universit\"at T\"ubingen 
\hfill {\tt daniele.agostini@uni-tuebingen.de}

\noindent Claudia Fevola, MPI-MiS Leipzig and Université Paris-Saclay, Inria {\em (current)} 
\hfill {\tt claudia.fevola@inria.fr}

\noindent Anna-Laura Sattelberger,  MPI-MiS Leipzig and KTH Stockholm {\em (current)}
\hfill {\tt alsat@kth.se}

\noindent Simon Telen, MPI-MiS Leipzig {\em (current)} and CWI Amsterdam 
\hfill {\tt simon.telen@mis.mpg.de}

\smallskip

\noindent Saiei-Jaeyeong Matsubara-Heo, Kumamoto University\hfill {\tt saiei@educ.kumamoto-u.ac.jp}

\end{document}